\documentclass{amsart}
\usepackage{amsmath}
\usepackage{amsfonts}
\usepackage{amssymb}
\usepackage{amstext}
\usepackage{graphicx}
\providecommand{\U}[1]{\protect \rule{.1in}{.1in}}
\providecommand{\U}[1]{\protect \rule{.1in}{.1in}}
\providecommand{\U}[1]{\protect \rule{.1in}{.1in}}
\newtheorem{theorem}{Theorem}[section]
\newtheorem{lemma}{Lemma}[section]
\newtheorem{proposition}{Proposition}[section]
\newtheorem{corollary}{Corollary}[section]

\newtheorem{definition}{Definition}[section]

\numberwithin{equation}{section}

\theoremstyle{remark}
\newtheorem{remark}{Remark}[section]
\numberwithin{equation}{section}
\begin{document}
\title[CR Poincar\'{e}-Lelong equation and Yamabe steady solitons]{On the CR Poincar\'{e}-Lelong equation, Yamabe steady solitons and structures
of complete noncompact Sasakian manifolds}
\author{$^{\star}$Der-Chen Chang}
\address{Department of Mathematics and Statistics, Georgetown University, Washington D.
C. 20057, USA\\
Department of Mathematics, Fu Jen Catholic University, Taipei 242, Taiwan, R.O.C.}
\email{chang@georgetown.edu}
\author{$^{\ast \ast}$Shu-Cheng Chang}
\address{Department of Mathematics and Taida Institute for Mathematical Sciences
(TIMS), National Taiwan University, Taipei 10617, Taiwan, R.O.C., Current address : Yau Mathematical Sciences
Center, Tsinghua University, Beijing, China
}
\email{scchang@math.ntu.edu.tw }
\author{$^{\dag}$Yingbo Han}
\address{{School of Mathematics and Statistics, Xinyang Normal University}\\
Xinyang,464000, Henan, P.R. China}
\email{{yingbohan@163.com}}
\author{Chien Lin}
\address{Department of Mathematics, National Tsing Hua University, Hsinchu 30013,
Taiwan, R.O.C}
\email{r97221009@ntu.edu.tw}
\thanks{$^{\star}$Der-Chen Chang is partially supported by an NSF grant DMS-1408839
and a McDevitt Endowment Fund at Georgetown University.}
\thanks{$^{\ast \ast}$Shu-Cheng Chang is partially supported in part by the MOST of Taiwan.}
\thanks{$^{\dag}$Yingbo Han is partially supported by an NSFC grant No. 11201400,
Nanhu Scholars Program for Young Scholars of {Xinyang Normal University} and
the Universities Young Teachers Program of Henan Province (2016GGJS-096),
China Scholarship Council(201508410400).}

\begin{abstract}
In this paper, we solve the so-called CR Poincar\'{e}-Lelong equation by solving the CR Poisson equation on
a complete noncompact CR $(2n+1)$-manifold
with nonegative pseudohermitian bisectional curvature tensors and vanishing
torsion which is an odd dimensional counterpart of K\"{a}hler geometry.
With applications of this solution plus the CR Liouvelle property, we study the structures of
complete noncompact Sasakian manifolds and CR Yamabe steady solitons.

\end{abstract}
\maketitle

\section{Introduction}

It is conjectured (\cite{gw1}, \cite{s} and \cite{y}) that a complete
noncompact Kahler manifold of complex dimension $m$ with positive holomorphic
bisectional curvature is biholomorphic to $\mathbf{C}^{m}$. \ It is the very
first result concerning such a conjecture was obtained by Mok-Siu-Yau
(\cite{msy}) and Mok (\cite{mok2}). More precisely, they proved that a
complete noncompact K\"{a}hler manifold of nonnegative holomorphic bisectional
curvature $M$ is isometrically biholomorphic to $\mathbf{C}^{m},$ $m\geq2$
under the assumptions of the maximum volume growth condition
\[
V_{o}\left(  r\right)  \geq \delta r^{2m}%
\]
for some point $o\in M,\  \delta>0,$ $r(x)=d(o,x)$ and the scalar curvature $R$
decays as
\[
R(x)\leq \frac{C}{1+r^{2+\varepsilon}},\  \  \ x\in M
\]
for $C>0$ and any arbitrarily small positive constant $\varepsilon$. A
Riemannian version of (\cite{msy}) was proved in \cite{gw2} shortly
afterwards. Since then there are several further works aiming to prove the
optimal result and reader is referred to \cite{mok1}, \cite{ctz}, \cite{cz},
\cite{n2}, \cite{nt1}, \cite{nt2} and \cite{nst}. A key common ingredient used
in the previous works such as \cite{msy}, \cite{n2} and \cite{nt2} is to solve
the so-called Poincar\'{e}-Lelong equation%
\[
i\partial \overline{\partial}u=\rho,
\]
by first solving the Poisson equation
\[
\Delta u=trace(\rho)
\]
for a given Ricci form $\rho$ of $M.$ They then applied the results to study
the analytic and geometric properties of $M$. \ Note that we refer to
\cite{n1} and \cite{cf} for adapting a different method which has also
succeeded in the recent resolution of the fundamental gap conjecture.

A CR $(2n+1)$-manifold of vanishing pseudohermitian torsion (i. e. Sasakian
manifold) is an odd dimensional counterpart of K\"{a}hler geometry, then it is
very natural to concerned with an CR analogue of Mok-Siu-Yau type theorem of
Poincare-Lelong equation in a complete noncompact strictly pseudoconvex CR
$(2n+1)$-manifold of vanishing pseudohermitian torsion with nonnegative
pseudohermitian bisectional curvature.

Let $(M,J,\theta)$ be a complete noncompact strictly pseudoconvex CR
$(2n+1)$-manifold (see in the appendix for basic notions). A piecewise smooth curve $\gamma
:[0,1]\rightarrow M$ is said to be horizontal if $\dot \gamma(t)\in \mathcal D$
whenever $\dot \gamma(t)$ exists. Here $\mathcal D$ is the linear span of the tangential vector fields in $\left \{Z_{\alpha},Z_{\bar{\alpha}}\right \}_{\alpha=1}^n $ which is a $2n$-dimensional vector space.
We denote $C_{p,q}$ to be the set of all
horizontal curves joining $p$,  $q$. Since $M$ is strictly pseudoconvex, the Levi form is positive definite. This implies that
${\mathcal D}\cup [\mathcal D,\mathcal D]=TM$, the tangent bundle of $M$. Therefore, $C_{p,q}\not=\emptyset$ by Chow's theorem (see {\it e.g.,} \cite {chow} and \cite {1}).  The Carnot-Carath\'{e}odory distance
between two points $p,\ q\in M$ is defined by
\[
d_{cc}(p,q)=\inf \left \{  l(\gamma):\  \gamma \in C_{p,q}\right \}  ,
\]
where the length of $\gamma$ is
\[
l(\gamma)=\int_{0}^{1}\left \langle \dot \gamma(t),\dot \gamma(t)\right \rangle
_{L_{\theta}}^{\frac{1}{2}}dt.
\] Now given any
nonnegative function $f$ on $M$, we define
\[
k_{f}(x,t)=\frac{1}{V_{x}(t)}\int_{B_{x}(t)}fd\mu.
\]
Here $V_{x}(r)$ is the volume of the geodesic ball $B_{x}(r)$ with respect to
the Carnot-Carath\'{e}odory distance $r(x,y)$ between $x$ and $y$, and
$r(x)=r(x,o)$ where $o\in M$ is a fixed point. Next we give the definition of the sublaplacian $\triangle_b$ (see the formula (\ref{7.1})).
\[
\begin{array}
[c]{c}
\Delta_{b}u=Tr\left(  (\nabla^{H})^{2}u\right)  =\sum_{\alpha}(u_{\alpha
\bar{\alpha}}+u_{\bar{\alpha}\alpha}).
\end{array}
\] We will follows the method of
solving Poisson equation as in \cite{nst}.

\begin{theorem}
\label{A1} Let $(M,J,\theta)$ be a complete noncompact strictly pseudoconvex
CR $(2n+1)$-manifold with nonnegative pseudohermitian Ricci curvature tensors
and vanishing torsion. Let $f$ be a nonnegative function and $k(t)=k_{f}%
(o,t)$, where $o\in M$ is a fixed point. Suppose that
\begin{equation}
\int_{0}^{\infty}k(t)dt<\infty \label{A1a}
\end{equation}
and suppose that there exist $1>\delta>0$, $h(t)\geq0$, $0\leq t\leq \infty$
with $h(t)=o(t)$ as $t\rightarrow \infty$ such that
\begin{equation}
\int_{0}^{t}sk(x,t)ds\leq h(t) \label{A1b}
\end{equation}
for all $x$ and for all $t\geq \delta r(x)$. Then the CR Poisson equation
\begin{equation}
\triangle_{b}u=f \label{A1c}
\end{equation}
\  \ has a solution such that for all $1>\varepsilon>0$
\begin{equation}
\begin{array}
[c]{ccl}
\alpha_{1}r\int_{2r}^{\infty}k(t)dt+\beta_{1}\int_{0}^{2r}tk(t)dt & \geq & u(x)\\
&  \geq & -\alpha_{2}r\int_{2r}^{\infty}k(t)dt-\beta_{2}\int_{0}^{\varepsilon
r}tk(x,t)dt+\beta_{3}\int_{0}^{2r}tk(t)dt
\end{array}
\end{equation}
for some positive constants $\alpha_{1}$, $\beta_{i}$, $1\leq i\leq3.$ which
depend on $n$ and $\alpha_{2}$ which depends on $n,\varepsilon$.
\end{theorem}

The main consequence is that if $f$ decays faster than $r^{-1}$ in the
following sense, then (\ref{A1c}) has a solution.

\begin{corollary}
\label{C12} Let $(M,J,\theta)$ be a complete noncompact strictly pseudoconvex
CR $(2n+1)$-manifold with nonnegative pseudohermitian Ricci curvature tensors
and vanishing torsion. Let $f$ be a nonnegative function and $k(t)=k_{f}
(o,t)$, where $o\in M$ is a fixed point. Suppose that
\[
\int_{0}^{\infty}k(t)dt<\infty
\]
and
\[
\sup_{\partial B_{o}(r)}f=o(r^{-1})
\]
as $r\rightarrow \infty.$ Then the CR Poisson equation
\[
\triangle_{b}u=f
\]
\  \ has a solution.
\end{corollary}

By deriving some basic results for solutions to the CR heat equation on
complete noncompact pseudohermitian $(2n+1)$-manifolds, we are able to solve
the CR Poincar\'{e}-Lelong equation with the help of Theorem \ref{A1}. \

The key is, following from Corollary \ref{c31} with $f=trace(\rho)$ and
$f_{0}(x)=0$ ,  that $u_{0}(x)=0$ and then (\ref{pl53}) ( in section $4$)
reduces to the CR Poisson equation (\ref{A1c}).

\begin{theorem}
\label{A2} Let $(M,J,\theta)$ be a complete CR $(2n+1)$-manifold with
nonnegative pseudohermitian bisectional curvature and vanishing torsion. Let
$\rho$ be a real closed $(1,1)$ form with the nonnegative trace $f$ and
$\label{54}\rho_{\alpha \overline{\beta},0}=0$, for $\alpha,\beta=1,\cdots,n$.
Assume that
\begin{equation}
\label{100}\int_{0}^{\infty}\frac{1}{V(B_{0}(t))}\int_{B_{0}(t)}||\rho||d\mu
dt<\infty
\end{equation}
and
\begin{equation}
\label{106}\lim \inf_{r\rightarrow \infty}\Big [\exp(-ar^{2})\int_{B_{0}(r)}%
||\rho||^{2}d\mu\Big ]<\infty \text{,\  \textrm{\  \ for some} }a>0.
\end{equation}
Then there is a solution of the CR Poincar\'{e}- Lelong equation
\[
i\partial_{b}\overline{\partial}_{b}u=\rho.
\]
Moreover for any $0<\varepsilon<1$, $u$ satisfies
\begin{equation}%
\begin{array}
[c]{ccl}
\alpha_{1}r\int_{2r}^{\infty}k(t)dt+\beta_{1}\int_{0}^{2r}tk(t)dt & \geq  & u(x)\\
&  \geq & -\alpha_{2}r\int_{2r}^{\infty}k(t)dt-\beta_{2}\int_{0}^{\varepsilon
r}tk(x,t)dt+\beta_{3}\int_{0}^{2r}tk(t)dt
\end{array}
\label{40}%
\end{equation}
for some positive constants $\alpha_{1}$, $\beta_{i}$, $1\leq i\leq3.$ which
depend on $n$ and $\alpha_{2}$ which depends on $n,\varepsilon$.
\end{theorem}

With its applications plus the CR Liouvelle theorem, we have

\begin{corollary}
\label{C1} Let $(M,J,\theta)$ be a complete CR $(2n+1)$-manifold with
nonnegative pseudohermitian bisectional curvature and vanishing torsion. Let
$\rho$ be a nonnegative real closed $(1,1)$ form and $\rho_{\alpha
\overline{\beta},0}=0$. Then $\rho \equiv0$ \ if the $\rho$ satisfies

(i)
\begin{equation}
\int_{0}^{r}s\Big [\frac{1}{V(B_{0}(t))}\int_{B_{0}(t)}||\rho||d\mu\Big ]ds=o(\log r),
\label{74}
\end{equation}
(ii)
\begin{equation}
\lim \inf_{r\rightarrow \infty}\Big [\exp(-ar^{2})\int_{B_{0}(r)}||\rho||^{2}
d\mu\Big ]<\infty \text{.} \label{75}
\end{equation}

\end{corollary}

Finally, we are able to investigate the geometry and classification of
$(2n+1)$-dimensional CR Yamabe solitons $\left(  M^{2n+1},\theta,f,\mu \right)
$ satisfying the soliton equation which is a special class of solutions to the CR Yamabe flow
\begin{equation}
\left \{
\begin{array}
[c]{l}
\frac{\partial}{\partial t}\theta(t)=-2W(t)\theta \left(  t\right)  ,\\
\theta \left(  0\right)  =\mathring{\theta},
\end{array}
\right.  \label{1b}
\end{equation}
on a pseudohermitian $(2n+1)$-manifold $(M^{2n+1},\mathring{\theta})$ given by
(cf \cite{ccc}), where $W(t)$ is the Tanaka-Webster scalar curvature with
respect to the evolving contact form $\theta \left(  t\right)  $. Moreover,
this special class of solutions to the CR Yamabe flow (\ref{1b}) is given by
self-similar solutions, whose contact forms $\theta_{t}$ deform under the CR
Yamabe flow only by a scaling function depending on $t$ and reparametrizations
by a $1$-parameter family of contact diffeomorphisms; meanwhile, the CR
structure $J$ shall be invariant under these diffeomorphisms.

\begin{definition}
(\cite{ccc})(i) We call $\left(  M^{2n+1},\theta,f,\mu \right)  $ a complete
$(2n+1)$-dimensional CR Yamabe soliton if for a family $\theta_{(t)}$ of
contact forms on $(M,\mathring{\theta})$ evolving by the CR Yamabe flow
(\ref{1b}) \ on $M\times \lbrack0,T)$ with the maximal time $T$ and for any
pairs $t_{1}<t_{2}$ $,$ the contact forms $\theta_{(t_{1})}$ and
$\theta_{(t_{2})}$ differ only by a contact diffeomorphism $\Phi$ with
$\theta_{(t_{2})}=\rho(t)\Phi^{\ast}\theta_{(t_{1})}$ for some smooth function
$\rho(t)$ with $\rho(t)>0$ and $\rho(0)=1.$ Furthermore, one can have a smoth
function $f$ and a constant $\mu=-\frac{1}{2}\rho^{\prime}(0)$ such that%
\[
\left \{
\begin{array}
[c]{lcl}
W\theta+\frac{1}{2}\emph{L}_{X_{f}}\theta & = & \mu \theta,\\
f_{\alpha \alpha}+iA_{\alpha \alpha}f & = & 0,
\end{array}
\right.
\]
Here $\emph{L}_{X_{f}}$ denotes Lie derivative by $X_{f}$ with $L_{X_{f}%
}J\equiv2(f_{\alpha \alpha}+iA_{\alpha \alpha}f)\theta^{\alpha}\otimes
Z_{\bar{\alpha}}+2(f_{\bar{\alpha}\bar{\alpha}}-iA_{\bar{\alpha}\bar{\alpha}%
}f)\theta^{\bar{\alpha}}\otimes Z_{\alpha}\  \mathrm{mod}\  \theta$ and
$X_{f}=if_{\alpha}Z_{\overline{\alpha}}-if_{\overline{\alpha}}Z_{\alpha
}-f\mathbf{T.}$ That is
\begin{equation}
\begin{array}
[c]{lcl}
f_{0}+2W & = & 2\mu
\end{array}
\label{2017-0}
\end{equation}
with $f_{\alpha \alpha}+iA_{\alpha \alpha}f=0.$ It is called \textit{shrinking}
if $\mu>0$, \textit{steady} if $\mu=0$, and \textit{expanding} if $\mu<0$.

(ii) A complete $(2n+1)$-dimensional CR Yamabe soliton is called a complete
gradient CR Yamabe soliton if there exists a smooth function $u$ on $M$ such
that
\begin{equation}
\Delta_{b}u=2(W-\mu) \label{2017}
\end{equation}
which is the same as
\[
\Delta_{b}u=-f_{0}.
\]
\end{definition}

\begin{remark}
\label{r2017} Let\textbf{\ }$(\mathbf{H}^{n},J,\theta,u,\mu)$ be a Heisenberg
$(2n+1)$-manifold with $W=0$ and $A_{\alpha \beta}=0.$ It is a Gaussian-type
\textit{gradient} CR Yamabe soliton for $u=\mu|z|^{2}$ with $\mu=-1,0,1.$
\end{remark}

In particular, every closed CR Yamabe soliton is a gradient CR Yamabe soliton
(\ref{2017-0}). In general, it follows from Corollary \ref{C12} that

\begin{corollary}
\label{C4} Any complete noncompact \textbf{steady }CR Yamabe soliton with
nonnegative pseudohermitian Ricci curvature tensors and vanishing torsion is a
complete \textbf{steady }gradient CR Yamabe soliton if
\[
\int_{0}^{\infty}k(t)dt<\infty
\]
and
\[
\sup_{\partial B_{o}(r)}W=o(r^{-1})
\]
as $r\rightarrow \infty.$ Here $W(x)$ is the nonnegative Tanaka-Webster scalar
curvature and
\[
k(x,t)=\frac{1}{V_{x}(t)}\int_{B_{x}(t)}Wd\mu.
\]
\end{corollary}

As a consequence of Corollary \ref{C1} with $\rho_{\alpha \overline{\beta}%
}=R_{\alpha \overline{\beta}}$, together with the Liouvelle property Lemma
\ref{T61}, we are able to investigate the structure of complete noncompact CR
$(2n+1)$-manifolds of nonnegative pseudohermitian bisectional curvature and
vanishing torsion (i.e. Sasakian manifolds) and complete \textbf{steady }CR
Yamabe solitons as well. Note that by CR Bianchi identity (\cite{l1}), we have
$R_{\alpha \overline{\beta},0}=0$ in a CR manifold with vanishing torsion.

\begin{corollary}
\label{C3} Let $(M^{2n+1},J,\theta,\varphi,\mu)$ be a complete noncompact CR
$(2n+1)$-manifold of nonnegative pseudohermitian bisectional curvature and
\textbf{vanishing torsion}. Assume that

\noindent (i)
\begin{equation}
\int_{0}^{r}s\Big (\frac{1}{V(B_{0}(t))}\int_{B_{0}(t)}||Ric||d\mu\Big )ds=o(\log r),
\label{2017-1}
\end{equation}

\noindent (ii)
\begin{equation}
\lim \inf_{r\rightarrow \infty}\Big [\exp(-ar^{2})\int_{B_{0}(r)}||Ric||^{2}
d\mu\Big ]<\infty \text{.} \label{2017-2}
\end{equation}
Then $M$ must be the CR flat space form. In particular any complete
\textbf{steady }CR Yamabe soliton of nonnegative pseudohermitian bisectional
curvature and vanishing torsion with (\ref{2017-1}) and (\ref{2017-2}) must be
the CR flat space form.
\end{corollary}

\begin{remark}
In the paper of \cite{ccc}, the second named author obtained the structure of
complete $3$-dimensional pseudo-gradient CR Yamabe solitons (shrinking, or
steady, or expanding) of vanishing torsion. \
\end{remark}

The rest of the paper is organized as follows. In section $2$, we will discuss
the condition on $f$ so that $\triangle_{b}u=f$ \ has a solution $u$ and we
also discuss the properties of $u$. In section $3,$ we derive some basic
results for solutions to the CR heat equation on complete noncompact
pseudohermitian $(2n+1)$-manifolds. In section $4,$ we solve the CR
Poincar\'{e}-Lelong equation by first solving the CR Poisson equation as in
section $2.$ In section $5$, we are able to investigate the structure of
complete noncompact CR $(2n+1)$-manifolds of nonnegative pseudohermitian
bisectional curvature and vanishing torsion (Sasakian) and complete
\textbf{steady }CR Yamabe solitons as well.

\begin{remark}
In this paper, the CR heat equation is for the sublaplacian
$$
\triangle_b\,=\,\sum_{\alpha=1}^n \big(Z_{\alpha}Z_{\bar{\alpha}}+Z_{\bar{\alpha}}Z_{\alpha}\big )
$$
not for the Kohn Laplacian $\Box_b$. However, the result for $\triangle_b$ can be extended easily to the
the Kohn Laplacian (at least on $(0,1$-forms) $\Box_b=\triangle_b+i\beta T$ with $|\beta|<1$.
Then we may use meromorphic extension to deal with other $\beta\in {\mathbf C}\setminus {\mathcal E}$ where ${\mathcal E}$ is a discrete exceptional set. See {\it e.g.,} \cite {ccfi}.
\end{remark}
\section{The CR Poisson Equation}

Let $(M,J,\theta)$ be a complete noncompact strictly pseudoconvex CR
$(2n+1)$-manifold with nonnegative pseudohermitian Ricci curvature tensors and
vanishing torsion. Let $f$ be a nonnegative function on $M$. In this section,
we will discuss the condition on $f$ so that $\triangle_{b}u=f$ \ has a
solution $u$ and we also discuss the properties of $u$.

Let $H(x,y,t)$ be the heat kernel of the CR heat equation on $M$, we define
the Green function
\[
G(x,y)=\int_{0}^{\infty}H(x,y,t)dt,
\]
if the integral on the right side converges. We can checks that $G$ is
positive and $\triangle_{b}G(x,y)=-\delta_{x}(y)$. It follows from the
properties of heat kernel and Volume double property (\cite{ccht} and
\cite{bbg}) that we obtain the following result :

\begin{theorem}
Let $(M,J,\theta)$ be a complete noncompact strictly pseudoconvex CR
$(2n+1)$-manifold with nonnegative pseudohermitian Ricci curvature tensors and
vanishing torsion. Then we have
\begin{equation}
C_{1}\int_{r^{2}(x,y)}^{\infty}V_{x}^{-1}(\sqrt{t})dt\leq G(x,y)\leq C_{2}
\int_{r^{2}(x,y)}^{\infty}V_{x}^{-1}(\sqrt{t})dt, \label{9}
\end{equation}
where $C_{1},C_{2}$ are two positive constants which depend $n$.
\end{theorem}

\begin{proof}
It follows from the definition of the Green function $G(x,y)$, we have
\[
G(x,y)=\int_{0}^{\infty}H(x,y,t)dt=\int_{0}^{r^{2}}H(x,y,t)dt+\int_{r^{2}
}^{\infty}H(x,y,t)dt
\]
a) If $t\geq r^{2}(x,y)$, it follows from \cite[Proposition 3.1.]{ccht} that
\begin{equation}
H(x,y,t)\leq C_{3}V^{-\frac{1}{2}}(B_{x}(\sqrt{t}))V^{-\frac{1}{2}}%
(B_{y}(\sqrt{t}))\exp(-C_{4}\frac{r^{2}(x,y)}{t}), \label{1}%
\end{equation}
where $C_{3},C_{4}$ are positive constants. Moreover, from the CR volume
doubling property (\cite[Corollary 1.1.]{ccht}),
\[
V(B_{x}(\sqrt{t}))\leq V(B_{y}(\sqrt{t}+r))\leq V(B_{y}(2\sqrt{t}))\leq
C_{5}2^{2C_{6}}V(B_{y}(\sqrt{t})),
\]
where $C_{5}$ is a positive constant which depends only on $n.$ Thus
\begin{equation}
V^{-\frac{1}{2}}(B_{y}(\sqrt{t}))\leq C_{5}^{\frac{1}{2}}2^{C_{6}}V^{-\frac
{1}{2}}(B_{x}(\sqrt{t})). \label{2}%
\end{equation}
It follows from (\ref{1}) and (\ref{2}) that
\begin{equation}
H(x,y,t)\leq C_{7}V^{-1}(B_{x}(\sqrt{t})) \label{3}
\end{equation}
and then
\begin{equation}
G(x,y)\leq \int_{0}^{r^{2}}H(x,y,t)dt+C_{7}\int_{r^{2}}^{\infty}V^{-1}
(B_{x}(\sqrt{t}))dt. \label{5}
\end{equation}
b) If $t\leq r^{2}(x,y)$, again from \cite[Proposition 3.1.]{ccht}, we have
\[
V(B_{x}(\sqrt{t}))\leq V(B_{y}(r+\sqrt{t}))\leq C_{5}(1+\frac{r}{\sqrt{t}
})^{2C_{6}}V(B_{y}(\sqrt{t})),
\]
that is,
\[
V^{-\frac{1}{2}}(B_{y}(\sqrt{t}))\leq C_{5}^{\frac{1}{2}}(1+\frac{r}{\sqrt{t}
})^{C_{6}}V^{-\frac{1}{2}}(B_{x}(\sqrt{t})).
\]
Thus
\begin{equation}
H(x,y,t)\leq C_{3}C^{\frac{1}{2}}(1+\frac{r}{\sqrt{t}})^{C_{6}}V^{-1}
(B_{x}(\sqrt{t}))\exp(-C_{4}\frac{r^{2}}{t}).
\end{equation}
Since $(1+\frac{r}{\sqrt{t}})\leq C_{8}\exp\big (\frac{C_{4}}{2}\frac{r^{2}}{t}\big )$,
where $C_{8}$ is a positive constant depending $C_{6}$, we have
\begin{equation}
H(x,y,t)\leq C_{9}V^{-1}(B_{x}(\sqrt{t}))\exp(-\frac{C_{4}}{2}\frac{r^{2}}
{t}), \label{4}
\end{equation}
where $C_{9}$ is a positive constant which depends only on $n$. \ Letting
$s=\frac{r^{4}}{t}$ with $r^{2}\leq s$, we have
\begin{equation}
\begin{array}
[c]{ccl}
\int_{0}^{r^{2}}H(x,y,t)dt & \leq & C_{9}\int_{0}^{r^{2}}V^{-1}(B_{x}(\sqrt
{t}))\exp(-\frac{C_{4}}{2}\frac{r^{2}}{t})dt\\
& = & C_{9}\int_{r^{2}}^{\infty}V^{-1}(B_{x}(\frac{r^{2}}{\sqrt{s}}
))\exp(-\frac{C_{4}}{2}\frac{s}{r^{2}})\frac{r^{4}}{s^{2}}ds.
\end{array}
\label{6}
\end{equation}
Again from \cite[Proposition 3.1.]{ccht}, we have
\[
V(B_{x}(\sqrt{s}))=V(B_{x}(\frac{s}{r^{2}}\frac{r^{2}}{\sqrt{s}}))\leq
C_{5}(\frac{s}{r^{2}})^{2C_{6}}V(B_{x}(\frac{r^{2}}{\sqrt{s}})),
\]
that is,
\begin{equation}
V^{-1}(B_{x}(\frac{r^{2}}{\sqrt{s}}))=V(B_{x}(\frac{s}{r^{2}}\frac{r^{2}%
}{\sqrt{s}}))\leq C_{5}(\frac{s}{r^{2}})^{2C_{6}}V^{-1}(B_{x}(\sqrt{s})).
\label{7}%
\end{equation}
From (\ref{6}) and (\ref{7}), we have
\begin{equation}
\begin{array}
[c]{ccl}
\int_{0}^{r^{2}}H(x,y,t)dt & \leq & C_{9}C_{5}\int_{r^{2}}^{\infty}
V^{-1}(B_{x}(\sqrt{s}))(\frac{r^{2}}{s})^{2-2C_{6}}\exp(-\frac{C_{4}}{2}
\frac{s}{r^{2}})ds.
\end{array}
\end{equation}
However, the function $x^{2C_{6}-2}\exp \big (-\frac{C_{4}}{2}x\big )$ is bound from
above. So we have
\begin{equation}
\int_{0}^{r^{2}}H(x,y,t)dt\leq C_{10}\int_{r^{2}}^{\infty}V^{-1}(B_{x}
(\sqrt{t}))dt, \label{8}
\end{equation}
where $C_{10}$ is a positive constant which depends only on $n$. Finally it
follow from (\ref{5}) and (\ref{8}) that
\[
G(x,y)\leq C_{2}\int_{r^{2}}^{\infty}V^{-1}(B_{x}(\sqrt{t}))dt,
\]
where $C_{2}$ is a positive constant which depends only on $n$. By using the
low bound of heat kernel, volume doubling property as in \cite{ccht} and
\cite{bbg}, we can also prove
\[
G(x,y)\geq C_{1}\int_{r^{2}}^{\infty}V^{-1}(B_{x}(\sqrt{t}))dt,
\]
where $C_{1}$ is a positive constant which depends only on $n$.
\end{proof}

\begin{theorem}
\label{91} Let $(M,J,\theta)$ be a complete noncompact strictly pseudoconvex
CR $(2n+1)$-manifold with nonnegative pseudohermitian Ricci curvature tensors
and vanishing torsion. Assume $(M,J,\theta)$ is nonparabolic and there is a
constant $\sigma>0$ such that the minimal Green's function $G(x,y)$ satisfies
\begin{equation}
\sigma^{-1}\frac{r^{2}(x,y)}{V(B_{x}(r(x,y)))}\leq G(x,y)\leq \sigma \frac
{r^{2}(x,y)}{V(B_{x}(r(x,y)))} \label{10}
\end{equation}
for all $x\neq y$ in $M$. Let $f$ be a nonnegative function and let
$k(x,t)=k_{f}(x,t)$ and $k(t)=k(o,t)$, where $o\in M$ is a fixed point.
Suppose that
\[
\int_{0}^{\infty}k(t)dt<\infty.
\]
Then the CR Poisson equation
\[
\triangle_{b}u=f
\]
has a solution $u$ such that for all $1>\varepsilon>0$,
\begin{equation}
\begin{array}
[c]{ccl}
\alpha_{1}\int_{2r}^{\infty}k(t)dt+\beta_{1}\int_{0}^{2r}k(t)dt & \geq & u(x)\\
&  \geq & -\alpha_{2}r\int_{2r}^{\infty}k(t)dt-\beta_{2}\int_{0}^{\varepsilon
r}tk(x,t)dt+\beta_{3}\int_{0}^{2r}tk(t)dt
\end{array}
\end{equation}
for some positive constants $\alpha_{1}$, $\alpha_{2}$, $\beta_{i}$, $1\leq
i\leq3.$
\end{theorem}

\begin{proof}
It follows from (\ref{9}) and (\ref{10}) that
\begin{equation}
C^{-1}\frac{r^{2}(x,y)}{V(B_{x}(r(x,y)))}\leq \int_{r}^{\infty}\frac{t}%
{V(B_{x}(t))}dt\leq C\frac{r^{2}(x,y)}{V(B_{x}(r(x,y)))},\label{11}%
\end{equation}
where $C$ is positive constant which depends on $n$ and $\sigma$. For all
$R>0$, let $G_{R}$ be the positive Green's function on $B_{o}(R)$ with zero
boundary value and denote
\[
u_{R}(x)=\int_{B_{o}(R)}(G_{R}(o,y)-G_{R}(x,y))f(y)dy.
\]
Then
\[
\triangle_{b}u=f
\]
in $B_{o}(R)$ and $u_{R}(o)=0$. For any $x$ with $r(x)=r$ with $R\gg r$, then
\begin{equation}
u_{R}(x)=I+II\label{12}
\end{equation}
with
\[
\begin{array}
[c]{ccc}
I & = & \int_{B_{o}(R)\backslash B_{o}(2r)}(G_{R}(o,y)-G_{R}(x,y))f(y)dy
\end{array}
\]
and
\[
\begin{array}
[c]{ccc}
II & = & \int_{B_{o}(2r)}(G_{R}(o,y)-G_{R}(x,y))f(y)dy.
\end{array}
\]
To estimate $I$, let $y$ be any point in $B_{o}(R)\backslash B_{o}(2r)$, then
$r_{1}=r(y)\geq2r=2r(x)$ and $r(z,y)\geq \frac{1}{2}r_{1}$ if $z\in B_{o}(r).$
Also $B_{o}(\frac{1}{2})\subseteq B_{o}(2r)$. \ Hence by the subgradient
estimate \cite{ckt},
\begin{equation}%
\begin{array}
[c]{ccl}
|G_{R}(o,y)-G_{R}(x,y)| & \leq & r\sup_{z\in B_{o}(r)}|\nabla_{b,z}%
G_{R}(z,y)|\\
& \leq & C_{1}\frac{r}{r_{1}}\sup_{z\in B_{o}(r)}G_{R}(z,y)\\
& \leq & C_{2}\frac{r}{r_{1}}G(o,y)\\
& \leq & C_{3}\frac{r}{r_{1}}\int_{r_{1}}^{\infty}\frac{t}{V(b_{o}(t))}dt,
\end{array}
\end{equation}
where $C_{1},C_{2},C_{3}$ are constant depending only on $n.$ Here we
used the Harnack inequality for $G_{R}(.,y)$ and the fact that $G_{R}(o,y)\leq
G(o,y)$.
\begin{equation}
\begin{array}
[c]{ccl}
|I| & \leq & C_{3}r\int_{B_{o}(R)\backslash B_{o}(2r)}\frac{1}{r(y)}%
(\int_{r(y)}^{\infty}\frac{t}{V(B_{o}(t))}dt)f(y)dy\\
& = & C_{3}r\int_{2r}^{R}t^{-1}(\int_{t}^{\infty}\frac{t}{V(B_{o}(t))}%
dt)(\int_{\partial B_{o}(t)}f(y)dy)dt\\
& = & C_{3}r[t^{-1}(\int_{t}^{\infty}\frac{t}{V(B_{o}(t))}dt)(\int_{B_{o}%
(t)}f(y)dy)|_{2r}^{R}\\
& &-  \int_{2r}^{R}(\int_{B_{o}(t)}f(y)dy)[-\frac{1}{t^{2}}\int_{t}^{\infty
}\frac{s}{V(B_{o}(s))}ds-\frac{1}{V(B_{o}(t))}]dt]\\
& \leq & C_{3}r[R^{-1}(\int_{R}^{\infty}\frac{R}{V(B_{o}(R))}dt)(\int
_{B_{o}(R)}f(y)dy)\\
& & + \int_{2r}^{R}(\int_{B_{o}(t)}f(y)dy)[\frac{1}{t^{2}}\int_{t}^{\infty
}\frac{s}{V(B_{o}(s))}ds+\frac{1}{V(B_{o}(t))}]dt]\\
& \leq & C_{4}r[Rk(R)+\int_{2r}^{R}k(t)dt]
\end{array}
\label{13}
\end{equation}
for some positive constants $C_{4}$ which depend on $n$ and $\sigma$. Here we
have used (\ref{11}). Moreover
\begin{equation}
\begin{array}
[c]{ccl}
\int_{B_{o}(2r)}G_{R}(o,y)f(y)dy & \leq & \int_{B_{o}(2r)}G(o,y)f(y)dy\\
& \leq & C_{5}\int_{0}^{2r}(\int_{t}^{\infty}\frac{s}{V(B_{o}(s))}%
ds)(\int_{\partial B_{o}(t)}f)dt\\
& = & C_{5}[(\int_{t}^{\infty}\frac{s}{V(B_{o}(s))}ds)(\int_{B_{o}(t)}%
f)|_{0}^{2r}\\
& &+  \int_{0}^{2r}\frac{t}{V(B_{o}(t))}(\int_{B_{o}(t)}f(y))dt]\\
& \leq & C_{5}[(\int_{2r}^{\infty}\frac{s}{V(B_{o}(s))}ds)(\int_{B_{o}%
(2r)}f)\\
&& +  \int_{0}^{2r}\frac{t}{V(B_{o}(t))}(\int_{B_{o}(t)}f(y))dt]\\
& \leq & C_{5}[C_{6}r^{2}k(2r)+\int_{0}^{2r}tk(t)dt]
\end{array}
\label{14}%
\end{equation}
for some positive constants $C_{5},C_{6}$ which depend on $n$ and $\sigma$.
From (\ref{12}), (\ref{13}) and (\ref{14}), we have
\begin{equation}
u_{R}(x)\leq C_{7}(rRk(R)+r^{2}k(2r)+r\int_{2r}^{R}k(t)dt)+\beta_{1}\int
_{0}^{2r}tk(t)dt\label{18}%
\end{equation}
for some positive constants $C_{7},\beta_{1}$ which depend on $n$ and $\sigma
$. As in the proof of (\ref{14}), using the lower bound of the Green's
function, we have
\begin{equation}
\int_{B_{o}(2r)}G(o,y)f(y)dy\geq C_{8}[C_{9}r^{2}k(2r)+\int_{0}^{2r}%
tk(t)dt]\label{17}%
\end{equation}
for some positive constants $C_{8},C_{9}$ which depend on $n$ and $\sigma$.
For $1>\varepsilon>0$,
\begin{equation}%
\begin{array}
[c]{ccl}
\int_{B_{x}(\varepsilon r)}G_{R}(x,y)f(y)dy & \leq & \int_{B_{x}(\varepsilon
r)}G(x,y)f(y)dy\\
& \leq & \beta_{2}[\int_{\varepsilon r}^{\infty}\frac{t}{V(B_{x}(t))}
dt(\int_{B_{x}(\varepsilon r)}fdy)\\
&& +  \int_{0}^{\varepsilon r}\frac{t}{V(B_{x}(t))}(\int_{B_{x}(t)}fdy)dt]\\
& \leq & C_{10}(\varepsilon r)^{2}k(x,\varepsilon r)+\beta_{2}\int_{0}^{\varepsilon r}tk(x,t)dt\\
& \leq & C_{11}r^{2}k((1+\varepsilon)r)+\beta_{2}\int_{0}^{\varepsilon
r}tk(x,t)dt
\end{array}
\label{15}
\end{equation}
for some positive constants $C_{10},\beta_{2}$ which depend on $n$ and
$\sigma$ and $C_{11}$ which depends on $n,\varepsilon$ and $\sigma$. Here we
used volume doubling property and the fact that $B_{x}(\varepsilon r)\subset
B_{o}((1+\varepsilon)r)$.
\begin{equation}
\begin{array}
[c]{ccl}
\int_{B_{o}(2r)\backslash B_{x}(\varepsilon r)}G_{R}(x,y)f(y)dy & \leq &
\int_{B_{o}(2r)\backslash B_{x}(\varepsilon r)}G(x,y)f(y)dy\\
& \leq & \sigma \frac{16r^{2}}{V_{x}(\varepsilon r)}\int_{B_{o}(2r)}f(y)dy\\
& \leq & C_{12}r^{2}k(2r)
\end{array}
\label{16}%
\end{equation}
for some positive constant $C_{12}$ which depends on $n,\varepsilon$ and
$\sigma$. From (\ref{12}), (\ref{13}), (\ref{17}), (\ref{15}) and (\ref{16}),
if $R\geq4r$, we have
\begin{equation}
\begin{array}
[c]{c}
u_{R}(x)\geq-C_{13}r(Rk(R)+\int_{2r}^{R}k(t)dt)-\beta_{2}\int_{0}^{\varepsilon
r}tk(x,t)dt+\beta_{3}\int_{0}^{2r}tk(t)dt
\end{array}
\label{19}
\end{equation}
for some positive constant $C_{13}$ which depend on $n,\varepsilon$ and
$\sigma$ and $\beta_{3}$ depends on $n$. Here we have used the fact that for
any $\alpha>1$, $k(\alpha R)\geq Ck(R)$ for some positive constant $C$ which
depends on $n,\alpha$ for all $R$. Since $\int_{0}^{\infty}k(t)dt<\infty$,
$\lim_{R\rightarrow \infty}Rk(R)=0$. Hence from (\ref{18}) and (\ref{19}),
$u_{R}(x)$ is bounded on compacts sets and there exists $R_{i}\rightarrow
\infty$ such that $u_{R_{i}}$ converges uniformly on compact sets to a
function $u$ which satisfies $\triangle_{b}u=f$. By (\ref{18}) and (\ref{19}),
let $R_{i}\rightarrow \infty$ we can conclude that $u$ satisfies the estimates
in the theorem.
\end{proof}

In order to prove \textbf{Theorem \ref{A1}}, we need one more lemma,

\begin{lemma}
\label{98} (\cite{nst}) Let $M=M_{1}\times M_{2}$, where $M_{1}$ and $M_{2}$
are complete noncompact manifolds with nonnegative Ricci curvature. Let
$f\geq0$ be a function on $M_{1}$ and be considered also as a function on $M$,
which is independent of the second variable. Let $x=(x_{1},x_{2})\in M$ and
$r>0$. Then
\begin{equation}
\frac{C^{-1}}{V(B_{x_{1}}^{(1)}(\frac{1}{\sqrt{2}}r))}\int_{B_{x_{1}}%
^{(1)}(\frac{1}{\sqrt{2}}r)}f\leq \frac{1}{V(B_{x}(r))}\int_{B_{x}(r)}%
f\leq \frac{C}{V(B_{x_{1}}^{(1)}(r))}\int_{B_{x_{1}}^{(1)}(r)}f
\end{equation}
for some constant $C>0$ depending only on the dimension of $M_{1}$ and $M_{2}%
$. Here $B_{x}(t)$, $B_{x_{1}}^{(1)}(t)$ are geodesic balls with radius $t$ in
$M,M_{1}$ with centers at $x,x_{1}$ respectively.
\end{lemma}

The proof of \textbf{Theorem \ref{A1} :}

\begin{proof}
Let $\widetilde{M}=M^{2n+1}\times \mathbb{R}^{4}$ with the standard Euclidean
space $\mathbb{R}^{4}$. Then $\widetilde{M}$ is nonparabolic. By volume
doubling property, we will prove that (\ref{11}) is true on $\widetilde{M}$.
\ It follows from \cite[Corollary 1.1.]{ccht} again that
\begin{equation}
V(B_{x}(t))=V(B_{x}(\frac{t}{r(x,y)}r(x,y)))\leq C_{v}(\frac{t}{r(x,y)}
)^{2C}V(B_{x}(r(x,y))),
\end{equation}
where $C_{v}$ is a positive constant which depends only on $n$, so we have
\[
V^{-1}(B_{x}(r(x,y)))C_{v}^{-1}(\frac{t}{r(x,y)})^{-2C}\leq V^{-1}(B_{x}(t)).
\]
That is
\begin{equation}
\frac{t}{V(B_{x}(r(x,y)))}C_{v}^{-1}(\frac{t}{r(x,y)})^{-2C}\leq \frac
{t}{V(B_{x}(t))}.\label{20}%
\end{equation}
On the other hand, we will have
\begin{equation}
\frac{t}{V(B_{x}(t))}\leq C_{v}(\frac{t}{r(x,y)})^{-2C}\frac{t}{V(B_{x}
(r(x,y)))}.\label{21}
\end{equation}
From (\ref{20}) and (\ref{21}), we have
\begin{equation}
\frac{t}{V(B_{x}(r(x,y)))}C_{v}^{-1}(\frac{t}{r(x,y)})^{-2C}\leq \frac
{t}{V(B_{x}(t))}\leq C_{v}(\frac{t}{r(x,y)})^{-2C}\frac{t}{V(B_{x}%
(r(x,y)))}\label{22}
\end{equation}
and then
\[
C_{1}^{-1}\frac{r^{2}(x,y)}{V(B_{x}(r(x,y)))}\leq \int_{r(x,y)}^{\infty}%
\frac{t}{V(B_{x}(t))}dt\leq C_{1}\frac{r^{2}(x,y)}{V(B_{x}(r(x,y)))},
\]
where $C_{1}$ is a positive constant depending only on $n$. From (\ref{9}), we
have
\begin{equation}
\sigma^{-1}\frac{r^{2}(x,y)}{V(B_{x}(r(x,y)))}\leq G(x,y)\leq \sigma \frac
{r^{2}(x,y)}{V(B_{x}(r(x,y)))},\label{23}%
\end{equation}
where $\sigma$ is a positive constant depending only on $n$. So we know that
$(\ref{10})$ is satisfied on $\widetilde{M}$. Then the rest proof is similar
to the proof as in \cite[Theorem 1.2.]{nst}, so we omit it here.
\end{proof}

\begin{corollary}
\label{c31} Let $(M,J,\theta)$ be a complete noncompact strictly pseudoconvex
CR $(2n+1)$-manifold with nonnegative pseudohermitian Ricci curvature tensors
and vanishing torsion. Let $f$ be a nonnegative locally H\"{o}lder continuous
function. Assume that $u$ is a solution of the CR Poisson equation
\[
\triangle_{b}u=f.
\]
If $f_{0}=0$, then
\[
u_{0}=0.
\]
\end{corollary}

\begin{proof}
It follows from \cite{ckl} that $\left[  \Delta_{b},\text{ }\mathbf{T}\right]
u=0$ if the torsion is vanishing that
\[
\triangle_{b}u_{0}=f_{0}.
\]
Then from Theorem \ref{A1}
\[
\begin{array}
[c]{ccl}
\alpha_{1}r\int_{2r}^{\infty}\widetilde{k}(t)dt+\beta_{1}\int_{0}
^{2r}t\widetilde{k}(t)dt &   \geq &u(x)\\
&   \geq  & -\alpha_{2}r\int_{2r}^{\infty}\widetilde{k}(t)dt-\beta_{2}\int
_{0}^{\varepsilon r}t\widetilde{k(}x,t)dt+\beta_{3}\int_{0}^{2r}t\widetilde
{k}(t)dt,
\end{array}
\]
where $\widetilde{k}(x,t)=\frac{1}{V_{x}(t)}\int_{B_{x}(t)}f_{0}d\mu$, and
$\widetilde{k}(t)=\widetilde{k}(o,t)$. If $f_{0}=0$, we can obtain that
$u_{0}=0$ as well.
\end{proof}

Observe that if $\int_{0}^{\infty}k(o,t)dt<\infty,$ then $\int_{0}^{\infty
}k(x,t)dt<\infty$ for all $x$. If $u$ is the solution obtained as in Theorem
\ref{A1}, then for any $x_{0}\in M$
\[
u(x)-u(x_{0})=\int_{M}(G(x_{0},y)-G(x,y))f(y)dy.
\]
From the proof of the theorem, it is easy to see that
\begin{equation}
\begin{array}
[c]{ccl}
\alpha_{1}r\int_{2r}^{\infty}k(x_{0},t)dt+\beta_{1}\int_{0}^{2r}
tk(x_{0},t)dt
&   \geq & u(x)-u(x_{0})\\
&  \geq &-\alpha_{2}r\int_{2r}^{\infty}k(x_{0},t)dt-\beta_{2}\int
_{0}^{\varepsilon r}tk(x,t)dt+\beta_{3}\int_{0}^{2r}tk(x_{0},t)dt,
\end{array}
\label{39}
\end{equation}
where $\alpha_{i}$ and $\beta_{j}$ are the constants in Theorem \ref{A1} and
$r=r(x,x_{0})$.

From (\ref{39}), we can obtain the following results.

\begin{theorem}
\label{49} The same assumptions and notation as in Theorem \ref{A1} or in
Theorem \ref{91}. Let $u$ be the solution of
\[
\triangle_{b}u=f
\]
obtained as in Theorem \ref{A1} or in Theorem \ref{91}. Then

(i)
\[
|\nabla_{b}u(x)|\leq C(n,\sigma)\int_{0}^{\infty}k(x,t)dt,
\]
where $C$ is a positive constant depending only on $n$.

(ii) For any $p\geq1$ and $\alpha \geq2$, if $u$ is a solution obtained in
Theorem \ref{91}, then
\[
\frac{1}{V_{o}(R)}\int_{B_{o}(R)}|\nabla_{b}u|^{p}dx\leq C(n,p)\Big (\int_{\alpha
R}^{\infty}k(t)dt\Big )^{p}+\frac{C(n,p,\alpha)R^{p}}{V(B_{o}(R))}\int
_{B_{o}(\alpha R)}f^{p}dx
\]
and if $u$ is a solution obtained in Theorem \ref{A1}, then
\[
\frac{1}{V_{o}(\frac{1}{\sqrt{2}}R)}\int_{B_{o}(\frac{1}{\sqrt{2}}R)}
|\nabla_{b}u|^{p}dx\leq C(n,p)\Big (\int_{\alpha R}^{\infty}k(t)dt\Big )^{p}
+\frac{C(n,p,\alpha)R^{p}}{V(B_{o}(R))}\int_{B_{o}(\alpha R)}f^{p}dx.
\]

(iii) If $f_{0}=0$ and $u$ is a solution obtained in Theorem \ref{91}, then
\[
\frac{1}{V_{o}(R)}\int_{B_{o}(R)}|(\nabla^{H})^{2}u|^{2}\leq C(n)\Big [R^{-2}
\big(\int_{4R}k(t)dt\big )^{2}+\frac{1}{V_{o}(4R)}\int_{B_{o}(4R)}f^{2}\Big  ]
\]
and if $u$ is a solution obtained in Theorem \ref{A1}, then
\[
\frac{1}{V_{o}(\frac{1}{\sqrt{2}}R)}\int_{B_{o}(\frac{1}{\sqrt{2}}R)}
|(\nabla^{H})^{2}u|^{2}\leq C(n,\sigma)\Big [R^{-2}\big (\int_{4R}k(t)dt\big )^{2}+\frac
{1}{V_{o}(4R)}\int_{B_{o}(4R)}f^{2}\Big ].
\]
\end{theorem}

\begin{proof}
Let us first consider the solution $u$ obtained in Theorem \ref{91}.

(i). From (\ref{39}), we can know that (i) is true.

(ii) From Theorem \ref{91}, for any $x\in B_{o}(R)$ and (\ref{10}) and the
subgradient estimate \cite{ckt}, we have
\begin{equation}%
\begin{array}
[c]{ccl}
|\nabla_{b}u| & \leq & \int_{M}|\nabla_{b,x}G(x,y)|f(y)dy\\
& \leq & C_{1}(n)\int_{M}r^{-1}(x,y)G(x,y)f(y)dy\\
& \leq & C_{2}(n,\sigma)\int_{M\backslash B_{o}(\alpha R)}\frac{r(x,y)}
{V(B_{x}(r(x,y)))}f(y)dy\\
&& +  C_{1}(n)\int_{B_{x}(\alpha R)}r^{-1}(x,y)G(x,y)f(y)dy
\end{array}
\label{92}
\end{equation}
for some positive constants $C_{1}(n)$ and $C_{2}(n,\sigma)$.
\begin{equation}
\begin{array}
[c]{ccl}
&  & \int_{M\backslash B_{o}(\alpha R)}\frac{r(x,y)}{V(B_{x}(r(x,y)))}f(y)dy\\
&  & \leq C_{3}(n)\int_{M\backslash B_{o}(\alpha R)}\frac{r(y)}{V(B_{o}
(r(y)))}f(y)dy\\
&  & =\int_{\alpha R}^{\infty}\frac{t}{V_{o}(t)}(\int_{\partial B_{o}
(r)}f)dt\\
&  & =C_{3}(n)\Big [\frac{t}{V_{o}(t)}\int_{B_{o}(t)}f|_{\alpha R}^{\infty}
-\int_{\alpha R}^{\infty}(\int_{B_{o}(t)}f)\frac{d}{dt}(\frac{t}{V_{o} (t)})dt\Big ]\\
&  & \leq C_{3}(n)\Big [-\int_{\alpha R}^{\infty}(\frac{1}{V_{o}(t)}-\frac
{tA_{o}(t)}{V_{o}(t)})(\int_{B_{o}(t)}f)dt\Big ]\\
&  & \leq C_{4}(n)\int_{\alpha R}^{\infty}\frac{1}{V_{o}(t)}(\int_{B_{o}
(t)}f)dt=C_{4}(n)\int_{\alpha R}^{\infty}k(t)dt
\end{array}
\label{93}
\end{equation}
for some constants $C_{3}(n),C_{4}(n)$, where we have used the fact that
$\alpha \geq2$, $tk(t)\rightarrow0$ as $t\rightarrow \infty$ and \cite[Corollary
1.1.]{ccht} $tA_{o}(t)\leq C_{5}(n)(n+3)V_{o}(t)$, where $C_{5}(n)$ is a
positive constant. Note that for any $z\in M$ and for $\rho>0$, we have
\begin{equation}
\int_{B_{z}(\rho)}\frac{r(z,y)}{V_{z}(r(z,y))}dy=\int_{0}^{\rho}\frac
{tA_{z}(t)}{V_{z}(t)}dt\leq C_{5}(n)(n+3)\rho.\label{94}
\end{equation}
Let us first assume $p>1$ and let $q=\frac{p}{p-1}$. From (\ref{92}) and
(\ref{93}), we have
\begin{equation}
\begin{array}
[c]{ccl}
\int_{B_{o}(R)}|\nabla_{b}u|^{p}&  \leq & C_{6}(n,\sigma,p)V_{o}
(R)(\int_{\alpha R}^{\infty}k(t)dt)^{p}\\
&  & +C_{7}(n,p)\int_{B_{o}(R)}(\int_{B_{o}{(\alpha R)}}r^{-1}
(x,y)G(x,y)f(y)dy)^{p}dx
\end{array}
\label{95}
\end{equation}
for some constants $C_{6}(n),C_{7}(n)$. From Theorem \ref{91}, for any $x\in
B_{o}(R)$ and (\ref{10}) and (\ref{94}), we have
\begin{equation}
\begin{array}
[c]{l}
\int_{B_{o}(R)}(\int_{B_{\alpha R}}r^{-1}(x,y)G(x,y)f(y)dy)^{p}dx\\
\leq q\int_{B_{o}(R)}(\int_{B_{o}{(\alpha R)}}r^{-1}(x,y)G(x,y)dy)^{\frac
{p}{q}}(\ int_{B_{o}{(\alpha R)}}r^{-1}(x,y)G(x,y)f^{p}(y)dy)dx\\
\leq \int_{B_{o}(R)}(\int_{B_{o}{(\alpha R)}}\frac{\sigma r(x,y)}%
{V_{x}(r(x,y))}dy)^{\frac{p}{q}}(\int_{B_{o}{(\alpha R)}}r^{-1}%
(x,y)G(x,y)f^{p}(y)dy)dx\\
\leq C_{8}(n,\sigma,p,\alpha)R^{\frac{p}{q}}\int_{B_{o}(R)}(\int
_{B_{o}{(\alpha R)}}r^{-1}(x,y)G(x,y)f^{p}(y)dy)dx\\
=C_{8}(n,\sigma,p,\alpha)R^{\frac{p}{q}}\int_{B_{o}(\alpha R)}(\int
_{B_{o}{(R)}}r^{-1}(x,y)G(x,y)dx)f^{p}(y)dy\\
\leq C_{9}(n,\sigma,p,\alpha)R^{p}\int_{B_{0}(\alpha R)}f^{p}(y)dy
\end{array}
\label{96}%
\end{equation}
for some constants $C_{8}(n,\sigma,p,\alpha)$ and $C_{9}(n,\sigma,p,\alpha)$.
Combine this with (\ref{95}), (ii) follows if $p>1$. The case $p=1$ can be
proved similarly.

(iii) By the assumption $f_{0}=0$, Corollary \ref{c31} and (\ref{96}), we have
$u_{0}=0$ and
\begin{equation}
\label{97}%
\begin{array}
[c]{ccl}%
\frac{1}{2}\triangle_{b}|\nabla_{b}u|^{2} & = & |(\nabla^{H})^{2}%
u|^{2}+\langle \nabla_{u},\nabla_{b}\triangle_{b}u\rangle+2Ric((\nabla
_{b}u)_{C},(\nabla_{b}u)_{C})\\
& \geq & |(\nabla^{H})^{2}u|^{2}+\langle \nabla_{b}u,\nabla_{b}\triangle
_{b}u\rangle=|(\nabla^{H})^{2}u|^{2}+\langle \nabla_{b}u,\nabla_{b}f\rangle,
\end{array}
\end{equation}
where we have used the $Ric\geq0$. Let $\eta \geq0$ be a smooth function with
compact in $B_{o}(2R)$. Multiplying (\ref{97}) by $\eta^{2}$ and integrating
by parts, we have
\[%
\begin{array}
[c]{ccl}
&  & \int_{B_{o}(2R)}\eta^{2}|(\nabla^{H})^{2}u|^{2}\leq \int_{B_{o}(2R)}%
[\frac{1}{2}\eta^{2}\triangle_{b}|\nabla_{b}u|^{2}-\eta^{2}\langle \nabla
_{b}u,\nabla_{b}f\rangle]\\
&  & \leq \int_{B_{o}(2R)}\eta^{2}f^{2}+\int_{B_{o}(2R)}\eta|\nabla_{b}%
\eta||\nabla_{b}u||f|+2\int_{B_{o}(2R)}\eta|\nabla_{b}\eta||\nabla_{b}%
|\nabla_{b}u|^{2}|\\
&  & \leq C[\int_{B_{o}(2R)}\eta^{2}f^{2}+\int_{B_{o}(2R)}|\nabla_{b}\eta
|^{2}|\nabla_{b}u|^{2}+\int_{B_{o}(2R)}\eta|\nabla_{b}\eta||\nabla
_{b}u||(\nabla^{H})^{2}u|]\\
&  & \leq C[\int_{B_{o}(2R)}\eta^{2}f^{2}+(1+\frac{1}{\varepsilon})\int
_{B_{o}(2R)}|\nabla_{b}\eta|^{2}|\nabla_{b}u|^{2}+\varepsilon \int_{B_{o}%
(2R)}\eta^{2}|(\nabla^{H})^{2}u|^{2}]
\end{array}
\]
for any $\varepsilon>0$, some absolute constant $C$. Choose $\varepsilon
=\frac{1}{2C}$, we have
\[
\int_{B_{o}(R)}|(\nabla^{H})^{2}u|^{2} \leq \widetilde{C}[\int_{B_{o}(2R)}%
f^{2}+R^{-2}\int_{B_{o}(2R)}|\nabla_{b}u|^{2}]
\]
for all $\eta \geq0$ with compact support in $B_{o}(R)$. Choose a suitable
$\eta$, we have
\[
\int_{B_{o}(2R)}\eta^{2}|(\nabla^{H})^{2}u|^{2} \leq2C[\int_{B_{o}(2R)}%
\eta^{2}f^{2}+\int_{B_{o}(2R)}|\nabla_{b}\eta|^{2}|\nabla_{b}u|^{2}]
\]
for some absolute constant $\widetilde{C}$. Combining this with (ii) the
results follows.

Suppose the assumption of Theorem \ref{A1} are satisfied. By the proof Theorem
\ref{A1}, Lemma \ref{98} and the above proof, we can obtain the other results
are also true.
\end{proof}

\section{The CR Heat Equation}

In this section, we will derive some basic results for solutions to the CR
heat equation on complete noncompact pseudohermitian $(2n+1)$-manifolds. We
refer to \cite{cf}, \cite{chl}, \cite{ccht}, \cite{bbg} for some basic properties.

We first recall a basic theorem for solutions to the CR heat equation from
\cite{chl} and \cite{nt1}.

\begin{theorem}
\label{34} Let $(M,J,\theta)$ be a complete noncompact strictly pseudoconvex
CR $(2n+1)$-manifold with nonnegative pseudohermitian Ricci curvature tensors
and vanishing torsion. Let $u$ be a continuous function on $M$ such that
\[
\frac{1}{V(B_{o}(r))}\int_{B_{o}(r)}|u|(x)dx\, \leq\,  \exp\big (ar^{2}+b\big )
\]
for some positive constant $a>0$ and $b>0$. Then the following initial value
problem
\begin{equation}
\left \{
\begin{array}
[c]{l}
(\frac{\partial}{\partial t}-\triangle_{b})v(x,t)=0,\\
v(x,0)=u(x).
\end{array}
\right.  \label{32}
\end{equation}
has a solution on $M\times(0,\frac{C}{16a}]$. Moreover,
\[
v(x,t)=\int_{M}H(x,y,t)u(y)dy,
\]
where $H(x,y,t)$ is the CR heat kernel of $M$.
\end{theorem}

For later proposes, we need more basic results from Theorem \ref{34} and \cite{nt2}.

\begin{lemma}
\label{35} Let $(M,J,\theta)$ be a complete noncompact strictly pseudoconvex
CR $(2n+1)$-manifold with nonnegative pseudohermitian Ricci curvature tensors
and vanishing torsion. Let $u,v$ be as in Theorem \ref{34}. Then for any
$0<\varepsilon<1$, there exists a constant $C=C(n,\varepsilon,a,b)$ and
$0<T_{0}<\frac{C(\varepsilon)}{16a}$ such that for all $(x,t)\in
M\times(0,T_{0}]$ with $r^{2}(x)\geq T_{0}$,
\[
|v(x,t)-\int_{B_{x}(\varepsilon r)}H(x,y,t)u(y)|\leq C(n,\varepsilon,a,b),
\]
where $r=r(x)$.
\end{lemma}

\begin{corollary}
\label{87} With the same assumption and notations as in Lemma \ref{35}. Then
for $x\in M$ with $r=r(x)\geq \sqrt{T_{0}}$ such that $u\geq0$ on
$B_{x}(\varepsilon r)$, then for any $0\leq t<T_{0}$
\[
-C(n,\varepsilon,a,b)+C(n,\varepsilon)\inf_{B_{x}(\varepsilon r)}u\leq
v(x,t)\leq C(n,\varepsilon,a,b)+\sup_{B_{x}(\varepsilon r)}u
\]
for some positive constant $C(n,\varepsilon)$.
\end{corollary}

Suppose that $|u(x)-u(y)|\leq \beta r(x,y)$, then, by Theorem \ref{34}, $v$ is
defined for all $t$. We have the following.

\begin{lemma}
Let $(M,J,\theta)$ be a complete noncompact strictly pseudoconvex CR
$(2n+1)$-manifold with nonnegative pseudohermitian Ricci curvature tensors and
vanishing torsion. Suppose that $u_{0}(x)=0$ and $|u(x)-u(y)|\leq \beta r(x,y)$
for all $x,y\in M$ and let $v$ be the solution of the heat equation with
initial value $u$ obtained in Theorem \ref{34}. Then for all $t>0$,
\[
\sup_{x\in M}|\nabla_{b}v(x,t)|\leq \beta.
\]
\end{lemma}

\begin{lemma}
\label{50} Let $(M,J,\theta)$ be a complete noncompact strictly pseudoconvex
CR $(2n+1)$-manifold with nonnegative pseudohermitian Ricci curvature tensors
and vanishing torsion. Let $u$ be a smooth function on $M$ with bounded
gradient and let $v$ be the solution of the heat equation initial value $u$ as
in (\ref{32}) with $u_{0}(x)=0$. Then for any $t>0$
\[
\sup_{M}|\nabla_{b}v(.,t)|\leq \sup_{M}|\nabla_{b}u|.
\]
\end{lemma}

\section{Poincar\'{e}-Lelong Equation}

In this section, assume that $\rho$ is a given real closed $(1,1)$-form on $M$
with $\rho_{\alpha \overline{\beta},0}=0.$ We will solve the CR Poincare-
Lelong equation
\begin{equation}
i\partial_{b}\overline{\partial}_{b}u=\rho \label{pl51}
\end{equation}
by first solving the CR Poisson equation
\begin{equation}
\Delta_{b}u=2f \label{pl52}
\end{equation}
for $f=trace(\rho).$ In fact, by commutation relation (\ref{A}), the CR
Poincare- Lelong equation (\ref{pl51}) implies
\begin{equation}
\triangle_{b}u+hu_{0}=2f, \label{pl53}
\end{equation}
where $h=\sum_{\alpha=1}^{n}h_{\alpha \overline{\alpha}}$. In the present
paper, it follows from Corollary \ref{c31} that $u_{0}(x)=0$ and then
(\ref{pl53}) reduces to the CR Poisson equation (\ref{pl52}).

Note that it follows from Theorem \ref{A1} that one can solve the CR Poisson
equation $\triangle_{b}u=2f$ \ with $u$ satisfying (\ref{40}). Moreover, $u$
is given by
\begin{equation}
\label{103}u(x)=2\int_{M}(G(o,y)-G(x,y))f(y)dy.
\end{equation}
Theorem \ref{A2} will follow from Lemma \ref{71} and Lemma \ref{72} below.

\begin{lemma}
\label{71} The same assumption as in Theorem \ref{A2}. Then

(i) The Cauchy problem (\ref{32}) with initial value $u$ has long time
solution $v(x,t)$ which is given by
\[
v(x,t)=\int_{M}H(x,y,t)u(y)dy
\]
and
\[
v_{0}(x,t)=0.
\]

(ii) There exists $t_{i}\rightarrow \infty$ such that $v(x,t_{i})-v(o,t_{i})$
together with their derivatives converge uniformly on compact subsets to a
constant function.
\end{lemma}

\begin{proof}
(i) We want to apply Theorem \ref{34}. For any $R>0$ and $x\in B_{o}(R)$,
\begin{equation}
\begin{array}
[c]{ccl}
|u(x)| & \leq & 2\int_{M}|G(o,y)-G(x,y)|f(y)dy \label{45}\\
& \leq & 2\int_{M\backslash B_{o}(4r)}|G(o,y)-G(x,y)|f(y)dy\\
&& +  2\int_{B_{o}(4R)}|G(o,y)-G(x,y)|f(y)dy\\
& = & I(x)+II(x).
\end{array}
\end{equation}
It follow from the proof of (\ref{13}) and the condition (\ref{100}), we have
\begin{equation}
\label{102}I(x)\leq C(n)r(x)\int_{4R}^{\infty}k(s)ds\leq C(n)r(x),
\end{equation}
where $C(n)$ is a positive constant. On the other hand, we have
\begin{equation}
\begin{array}
[c]{ccl}
\int_{B_{o}(R)}II(x)dx & \leq & \int_{x\in B_{o}(R)}[2\int_{y\in B_{o}%
(4R)}(G(o,y)+G(x,y))f(y)dy]dx\nonumber \label{44}\\
& = & \int_{y\in B_{o}(4R)}[2\int_{B_{o}(R)}(G(o,y)+G(x,y))dx]f(y)dy
\end{array}
\end{equation}
and from the proof of (\ref{14}),
\begin{equation}%
\begin{array}
[c]{ccl}
\label{42} &  & \int_{y\in B_{o}(4R)}[\int_{B_{o}(R)}G(o,y)dx]f(y)dy\\
& = & 2V_{o}(R)\int_{B_{o}(4R)}G(o,y)f(y)dy\\
& \leq & C(n)V_{o}(R)[R^{2}k(4R)+\int_{0}^{4R}tk(t)dt].
\end{array}
\end{equation}
Moreover from \cite[Corollary 1.1.]{ccht}, we have
\[
tA_{o}(t)\leq C(n)V(B_{o}(t)).
\]
Thus
\begin{align}
&  2\int_{y\in B_{o}(4R)}[\int_{x\in B_{o}(R)}%
G(x,y)dx]f(y)dy\nonumber \label{43}\\
&  \leq C(n)\int_{y\in B_{o}(4R)}[\int_{x\in B_{o}(R)}\frac{r^{2}%
(x,y)}{V(B_{y}(r(x,y)))}dx]f(y)dy\nonumber \\
&  \leq C(n)\int_{y\in B_{o}(4R)}[\int_{x\in B_{y}(5R)}\frac{r^{2}%
(x,y)}{V(B_{y}(r(x,y)))}dx]f(y)dy\\
&  =C(n)\int_{y\in B_{o}(4R)}[\int_{0}^{5R}\frac{t^{2}A_{y}(t)}{V(B_{y}(t))}dt]f(y)dy\nonumber \\
&  \leq C(n)V(B_{o}(R))R^{2}k(4R).\nonumber
\end{align}
All together with (\ref{44}), (\ref{42}) and (\ref{43}), we have
\begin{equation}
\label{101}\int_{B_{o}(R)}II(x)dx\leq C(n)V(B_{o}(R))(R^{2}k(4R)+\int_{0}%
^{4R}sk(s)ds).
\end{equation}
Also from (\ref{45}), (\ref{102}), (\ref{101}) and (\ref{100}), we have
\[
\frac{1}{V(B_{o}(R))}\int_{B_{o}(R)}|u|\leq C(n)(1+R^{2}).
\]
Then (i) of Lemma follows easily from Theorem \ref{34}. Furthermore, since the
torsion is vanishing, it is easy to know that $v_{0}(x,t)$ is a solution of
the following system:
\begin{equation}
\left \{
\begin{array}
[c]{l}%
(\frac{\partial}{\partial t}-\triangle_{b})\mu(x,t)=0,\\
\mu(x,0)=0.
\end{array}
\right.
\end{equation}
It follows from Theorem \ref{34} again that
\[
v_{0}(x,t)=\mathbf{T}(v(x,t))=0
\]
on $M\times \lbrack0,\infty).$

(ii) Let us first give an estimate of $|\nabla_{b}v|$. We proceed as in proof
of Theorem \ref{34}. Namely, use cut-off functions $\varphi_{i}$ and denote
$f_{i}=\varphi_{i}f$. Solve
\[
\triangle_{b}u=2f_{i}
\]
by using (\ref{103}) and find the solution $v_{i}$ of (\ref{32}) with initial
value $u_{i}$. Then $v_{i}$ subconverge to $v$ together with their derivatives
uniformly on compact sets on $M\times \lbrack0,\infty)$. From Theorem \ref{49},
we know that $|\nabla_{b}u_{i}|$ is bounded, and hence $|\nabla_{b}v_{i}|$ is
bounded by Lemma \ref{50}. We can apply the maximum principle to $|\nabla
_{b}v_{i}|$ which is a subsolution of the CR heat equation and conclude that
for any $x$ such that $r(x)\leq \sqrt{t}$,
\begin{align}
|\nabla_{b}v_{i}|  &  \leq \int_{M}H(x,y,t)|\nabla_{b}u_{i}|dy\nonumber \\
& \leq C(n)\sup_{r\geq \sqrt{t}}\frac{1}{V(B_{x}(r))}\int_{B_{x}(r)}%
|\nabla_{b}u_{i}|(y)dy\nonumber \\
&  \leq C(n)\sup_{r\geq \sqrt{t}}\frac{1}{V(B_{o}(2r))}\int_{B_{o}(2r)}
|\nabla_{b}u_{i}|(y)dy\\
&  \leq C(n)\sup_{r\geq \sqrt{t}}(\int_{4r}^{\infty}k(s)ds+rk(4r))\nonumber \\
&  \leq C(n)\int_{4\sqrt{t}}^{\infty}k(s)ds\nonumber
\end{align}
for some positive constant $C(n)$. Here we have used \cite[(3.13)]{cf} and
Theorem \ref{49} and CR Volume doubling property as well as the fact that
$0\leq f_{i}\leq f$. Hence
\[
\sup_{x\in B_{o}(\sqrt{t})}|\nabla_{b}v_{i}|(x,t)\leq C(n)\int_{4\sqrt{t}}^{\infty}k(s)ds
\]
for all $i$ and then
\begin{equation}
\label{104}\sup_{x\in B_{o}(\sqrt{t})}|\nabla_{b}v|(x,t)\leq C(n)\int_{4\sqrt{t}}^{\infty}k(s)ds.
\end{equation}
On the other hand, $f_{i}$ has compact support, $u_{i}$ and $v_{i}$ are
bounded. Since $(v_{i})_{t}$ is a solution of the following system:
\begin{equation}
\left \{
\begin{array}
[c]{l}
(\frac{\partial}{\partial t}-\triangle_{b})\mu(x,t)=0,\\
\mu(x,0)=2f_{i}(x).
\end{array}
\right.
\end{equation}
as in the proof of Lemma \ref{50}, one can prove that for any $T>0$, there
exist constant $C_{i}(n)$ such that
\[
\int_{0}^{T}[\frac{1}{V(B_{o}(r))}\int_{B_{o}(r)}|(v_{i})_{t}|^{2}dy]dt\leq
C_{i}(n)
\]
for all $r$. Hence we can apply maximum principle and conclude that
\begin{align}
\frac{\partial v_{i}}{\partial t}  &  =2\int_{M}H(x,y,t)f_{i}(y)dy\nonumber \\
&  \leq C(n)\sup_{r\geq \sqrt{t}}\frac{1}{V(B_{x}(r))}\int_{B_{x}(r)}%
f_{i}(y)dy\nonumber \\
&  \leq C(n)\sup_{r\geq \sqrt{t}}k(x,t).\nonumber
\end{align}
Note that we also have $(v_{i})_{t}\geq0$. Hence
\begin{equation}
\label{52}0\leq \frac{\partial v}{\partial t}(x,t)\leq C(n)\sup_{r\geq \sqrt{t}%
}k(x,t)\leq \frac{C(n)}{\sqrt{t}}\int_{\sqrt{t}}^{\infty}k(x,s)ds.
\end{equation}
All and (\ref{100}), (\ref{104}) and (\ref{52}), for any $t_{0}>1$, and $r>0$,
the function $v(x,t)-v(0,t_{0})$ is bounded in $B_{0}(r)\in \lbrack
t_{0}-1,t_{0}+1]$ by a constant which is independent of $t_{0}$ and
\[
\lim_{t\rightarrow \infty}\sup_{B_{o}(r)}|\nabla_{b}v(.,t)|\rightarrow0.
\]
Hence we know that (ii) is true.
\end{proof}

Now we denote that
\[
\omega=u-v.
\]
\begin{lemma}
\label{72} Assume that $\rho$ satisfies the equality:
\[
\rho_{\alpha,\overline{\beta},0}=0
\]
for $\alpha,\beta=1,\cdots,n$. As $t\rightarrow \infty$, we have $||\rho
-\sqrt{-1}\partial_{b}\overline{\partial_{b}}\omega||$ converges to zero
uniformly on compact subsets of $M$.
\end{lemma}

\begin{proof}
We claim that
\begin{equation}
\label{53}||\rho-i\partial_{b}\overline{\partial_{b}}\omega||(x,t)\leq \int
_{M}H(x,y,t)||\rho||(y)dy.
\end{equation}
If (\ref{53}) is true, then one can using \cite[(3.13)]{cf} again to conclude
that, for $x\in B_{o}(\sqrt{t})$,
\begin{align}
||\rho-i\partial_{b}\overline{\partial_{b}}\omega||(x,t)  &  \leq
C(n)\sup_{r\geq \sqrt{t}}\frac{1}{V(B_{x}(r))}\int_{B_{x}(r)}||\rho
||(y)dy\nonumber \\
&  \leq C(n)\sup_{r\geq \sqrt{t}}\frac{1}{V(B_{x}(2r))}\int_{B_{x}(2r)}%
||\rho||(y)dy\nonumber
\end{align}
for some positive constant $C(n)$. From the assumption (\ref{100}), this
implies that
\[
\sup_{B_{o}(\sqrt{t})}||\rho-i\partial_{b}\overline{\partial_{b}}%
\omega||(.,t)\rightarrow0,
\]
as $t\rightarrow \infty$ and the lemma follows.

To prove (\ref{53}), we first observe that since $\rho$ is real closed
(1,1)-form and the torsion is vanishing, locally it can be written as
$i\partial_{b}\overline{\partial_{b}}\Psi$. Then we have
\begin{equation}
\triangle_{b}\Psi+ih\Psi_{0}=2f,\label{55}
\end{equation}
where $h=\sum_{\alpha=1}^{n}h_{\alpha \overline{\alpha}}$. From the equation
(\ref{54}), we know that the equation (\ref{55}) turns into the following
equation
\[
\triangle_{b}\Psi=2f.
\]
From the definition of $\omega$, we have
\[
(\triangle_{b}-\frac{\partial}{\partial t})(\Psi-\omega)(x,t)=0.
\]
Hence $\eta=\rho-i\partial_{b}\overline{\partial}_{b}\omega=i\partial
_{b}\overline{\partial}_{b}(\Psi-\omega)$ satisfies the following
Lichnerowicz--Laplacian equation
\begin{equation}
\frac{\partial}{\partial t}\eta_{\alpha \overline{\beta}}=\triangle_{b}
\eta_{\alpha \overline{\beta}}+2R_{\alpha \overline{\gamma}\mu \overline{\beta}
}\eta_{\gamma \overline{\mu}}-(R_{\gamma \overline{\beta}}\eta_{\alpha
\overline{\gamma}}+R_{\alpha \overline{\gamma}}\eta_{\gamma \overline{\beta}
}).\label{84}
\end{equation}
Here we use the CR Bochner-Weitzenb\"{o}ck formula as in \cite{ccf}. Now by
(\ref{100}), since $\omega=0$ at $t=0$, we can have the following inequality
\[
\int_{M}||\eta||(x,0)\exp(-ar^{2}(x))dx<\infty
\]
for any $a>0$. Next we estimate $|(\nabla^{H})^{2}v|$. \ By CR Bochner formula
(\cite{cc}), we have
\[
(\triangle_{b}-\frac{\partial}{\partial t})|\nabla_{b}v|^{2}=\triangle
_{b}|\nabla_{b}v|^{2}-\frac{\partial}{\partial t}|\nabla_{b}v|^{2}
\geq2|(\nabla^{H})^{2}v|^{2}+4\langle J\nabla_{b}v,\nabla_{b}v_{0}\rangle.
\]
From Theorem \ref{34} and vanishing of the torsion in assumptions in Theorem
\ref{A2}, we have $v_{0}=0$ for all $t\geq0$. so the inequality becomes the
following
\begin{equation}
(\triangle_{b}-\frac{\partial}{\partial t})|\nabla_{b}v|^{2}\geq2|(\nabla
^{H})^{2}v|^{2}.\label{62}
\end{equation}
For any $T>1$ and $r^{2}\geq T$, multiplying the inequality (\ref{62}) by a
suitable cut-off function and integrating by parts, we have
\begin{align}
&  \int_{0}^{T}[\frac{1}{V(B_{o}(r))}\int_{B_{o}(r)}|(\nabla^{H})^{2}
v|^{2}dx]dt\nonumber \label{63}\\
 \leq & C_{1}[\frac{1}{r^{2}}\int_{0}^{T}\frac{1}{V(B_{o}(2r))}\int
_{B_{o}(2r)}|\nabla_{b}v|^{2}dxdt+\frac{1}{V(B_{o}(2r))}\int_{B_{o}
(2r)}|\nabla_{b}u|^{2}dx].
\end{align}
Then we have
\begin{align}
&  \int_{0}^{T}\frac{1}{V(B_{o}(r))}\int_{B_{o}(r)}|(\nabla^{H})^{2}
v|^{2}dxdt\nonumber \label{64}\\
 \leq & C_{2}[(T+1)\frac{1}{V(B_{o}(8r))}\int_{B_{o}(8r)}|\nabla_{b}
u|^{2}dx\nonumber \\
&\,  +\int_{0}^{T}t^{-2}(\int_{8r}^{\infty}\exp(-\frac{C(\rho_{2})s^{2}}
{8t})s\frac{1}{V(B_{o}(s))}\int_{B_{o}(s)}|\nabla_{b}u|^{2}dyds)^{2}dt].
\end{align}
From Theorem \ref{49}, (\ref{100}) and (\ref{64}),
\begin{align}
&  \int_{0}^{T}\frac{1}{V(B_{o}(r))}\int_{B_{o}(r)}|(\nabla^{H})^{2}
v|^{2}dt\nonumber \label{65}\\
  \leq &\,  C_{3}[(T+1)\frac{1}{V(B_{o}(8r))}\int_{B_{o}(8r)}|\nabla_{b}
u|^{2}dx\nonumber \\
& \,  +\int_{0}^{T}(\int_{4r}^{\infty}\exp(-\frac{C(\rho_{2}s^{2})}{8t}
)d(\frac{s^{2}}{t}))^{2}dt]\\
 \leq & \, C_{4}(T+1)[\frac{1}{V(B_{o}(4r))}\int_{B_{o}(4r)}|\nabla_{b}
u|^{2}dx+1]\nonumber \\
 \leq & \, C_{5}(T+1)[(\int_{16r}^{\infty}k(s)ds)^{2}+r^{2}\frac{1}
{V(B_{o}(16r))}\int_{B_{o}(16r)}||\rho||^{2}dx+1]\nonumber \\
\leq & \, C_{6}(T+1)[r^{2}\frac{1}{V(B_{o}(8r))}\int_{B_{o}(8r)}||\rho
||^{2}dx+1],\nonumber
\end{align}
where we use Theorem \ref{49}. We can also obtain from Theorem \ref{49} that
\begin{equation}
\frac{1}{V(B_{o}(r))}\int_{B_{o}(r)}|(\nabla^{H})^{2}u|^{2}dx\leq C_{7}
[r^{2}\frac{1}{V(B_{o}(2r))}\int_{B_{o}(2r)}||\rho||^{2}dx+1]\label{105}
\end{equation}
for some constant $C_{7}$ depending only on $n$. \ From (\ref{65}) and
(\ref{105}), we have
\begin{align}
&  \int_{0}^{T}\frac{1}{V(B_{o}(r))}\int_{B_{o}(r)}||\rho-\sqrt{-1}
\partial \overline{\partial}\omega||^{2}dxdt\nonumber \label{68}\\
 \leq &\,  C(T+1)[r^{2}\frac{1}{V(B_{o}(2r))}\int_{B_{o}(2r)}||\rho||^{2}dx+1]
\end{align}
for some constant $C$ independent of $T$ and $r$. From (\ref{106}) and
(\ref{68}),
\begin{equation}
\lim_{r\rightarrow \infty}\inf \int_{0}^{T}\int_{B_{o}(r)}||\eta||^{2}
(x,t)\exp(-ar^{2})dxdt<\infty.\label{107}
\end{equation}
From (\ref{84}), we have
\[
(\triangle_{b}-\frac{\partial}{\partial t})||\eta||\geq0.
\]
Therefore we know that $F=\eta-\int_{M}H(x,y,t)||\eta||(y,0)dy$ is also a
subsolution of the CR heat equation. Finally, it follow from (\ref{107}) and
the proof of \cite[Theorem 1.2.]{nt2} and \cite[Lemma 4.5.]{ccf} that%
\[
||\eta||\leq \int_{M}H(x,y,t)||\eta||(y,0)dy
\]
and then (\ref{53}) holds.
\end{proof}

\section{Structures of Complete Noncompact Sasakian Manifolds}

In this section, \ as a consequence of Corollary \ref{C1}, we are able to
investigate the structure of complete noncompact CR $(2n+1)$-manifolds of
nonnegative pseudohermitian bisectional curvature and vanishing torsion
(Sasakain) and complete \textbf{steady }CR Yamabe solitons as well.

Let $\{U_{\alpha}\}$ be an open covering of $M$ and $\pi_{\alpha}:U_{\alpha
}\rightarrow V_{\alpha}\subset \mathbf{C}^{n}$ submersion. Since $M$ is a
Sasakian manifold (\cite{fow}), we have a transverse Kahler structure
$(V_{\alpha},g_{\alpha}^{T})$ as following : There is a canonical isomorphism
\[
d\pi_{\alpha}:D_{p}\rightarrow T_{\pi_{\alpha}(p)}V_{\alpha}
\]
for $D=Ker\theta \subset TM$ and $p\in U_{\alpha}$. Let $\left \{  z^{1,}
z^{2},...,z^{n}\right \}  $ be the local holomorphic coordinates on $V_{\alpha
}$ and $T_{1,0}(M)=(D\otimes \mathbf{C)}^{1,0}$ span by the vectors of the form
\[
Z_{i}=\left(  \frac{\partial}{\partial z^{i}}-\theta \left(  \frac{\partial
}{\partial z^{i}}\right)  T\right)
\]
for $i=1,2,...,n$. It is clear that
\[
d\theta(Z_{i},Z_{\overline{j}})=d\theta(\frac{\partial}{\partial z^{i}}
,\frac{\partial}{\partial z\overline{^{j}}}).
\]
Then the restriction of $d\theta$ to the slice in $U_{\alpha}$ is the Kaehler
form which is the same as the Kaehler metric $g_{\alpha}^{T}$ on $V_{\alpha}.$
By this expression, we know that $z_{i}f=\frac{\partial f}{\partial z^{i}}$
locally; in other words, we could really view a Sasakian manifold as the
disjoint union of the slices on which there are transverse Kahler structures.
As an example, $\left \{  Z_{j}=\frac{\partial}{\partial z^{j}}+i\overline
{z}^{j}\frac{\partial}{\partial t};\ T=\frac{\partial}{\partial t}\right \}  $
is exactly a local frame and $\theta=dt+i%
%TCIMACRO{\dsum \limits_{j}}%
%BeginExpansion
{\displaystyle \sum \limits_{j}}
%EndExpansion
\left(  z^{j}d\overline{z}^{j}-\overline{z}^{j}dz^{j}\right)  $ is a
pseudohermitian structure on a $(2n+1)$-dimensional Heisenberg group $H^{n}$.

Note that in our situation as the Corollary \ref{C1} and Lemma \ref{71}, the
CR plurisuharmonic function
\[
\partial_{b}\overline{\partial}_{b}u\geq0
\]
is the same of the usual plurisuharmonic function with respect to the
transverse Kahler structures on $V_{\alpha}$
\[
\partial \overline{\partial}u\geq0,
\]
if $u_{0}=0$ in a Sasakian manifold.

In this section, we first obtain the following Proposition which served as the
CR analogue of results as in \cite[Theorem 3.1.]{nt1}. 
\begin{proposition}
\label{76} Let $(M,J,\theta)$ be a complete CR $(2n+1)$-manifold with
nonnegative pseudohermitian bisectional curvature and vanishing torsion. Let
$u(x)$ be a continous CR plurisubharmonic function on $M$ with $u_{0}=0$.
Assume that
\begin{align}
\label{77}|u|(x)\leq C\exp\big (ar^{2}(x)\big )
\end{align}
for some positive constants $a$ and $C$. Let $v(x,t)$ be the solution to the
heat equation on $M\times \lbrack0,\frac{C_{0}}{16a}]$ with initial value $u$,
obtained by Theorem \ref{34}. Then there exists $T_{0}>0$ depending only on
$a$ and there exists $T_{0}>T_{1}>0$ such that the following is true:

\noindent (i) For $0<t\leq T_{0}$, $v(.,t)$ is a smooth plurisubharmonic function.

\noindent (ii) Let
\[
K(x,t)\,=\, \big \{ \omega \in T_{x}^{1,0}M)|\, v_{\alpha \overline{\beta}}(x,t)\omega
^{\alpha}=0,\,\text{for all}\,\, \,  \beta\big \}
\]
be the null space of $v_{\alpha \overline{\beta}}$. Then for $0<t<T_{1}$,
$K(x,t)$ is a distribution on $M$.
\end{proposition}

In order to give an application of Theorem \ref{A2}, we also need the
following CR Liouvelle property.

\begin{lemma}
\label{T61} Let $(M,J,\theta)$ be a complete CR $(2n+1)$-manifold with
nonnegative pseudohermitian bisectional curvature and vanishing torsion. Let
$u(x)$ be a continuous CR Plurisubharmonic function with $u_{0}=0$ on $M$.
Assume that
\begin{equation}
u(x)=o(\log r). \label{61}
\end{equation}
Then $u(x)$ must be a constant.
\end{lemma}

\begin{proof}
We may assume that $M$ is simply connected. First we let $u_{c}=\max \{u,c\}$
with $v_{c}\geq0$. By the assumption (\ref{61}), we can conclude that $u_{c}$
satisfies (\ref{77}) and $u_{c}$ is CR plurisubharmonic. Therefore, we can
solve the CR heat equation with $u_{c}(x)$ as the initial data. Denote the
solution by $v_{c}$ on $M\times \lbrack0,T_{0}]$. Applying Theorem \ref{76} (i)
we obtain that $v_{c}$ is CR plurisubharmonic. By Proposition \ref{76} (ii),
for any $t_{0}>0$ small enough, $D=D_{1}\times D_{2}$ isometrically and
$(v_{c})_{\alpha \bar{\beta}}$ is zero when restricted on $D_{1}$ and
$(v_{c})_{\alpha \bar{\beta}}$ is positive everywhere when restricted on
$D_{2}$. By Corollary \ref{87}, we have
\begin{align}
\label{88}v_{c}(x,t)=o(\log r(x))
\end{align}
Hence when restricted on $D_{2}$, (\ref{88}) is still true. This contradicts
the fact $(\partial_{b}\bar{\partial_{b}}v_{c})^{n}=0$ (\cite[Lemma 3.3.]%
{nt1}) and $(v_{c})_{\alpha \bar{\beta}}$ is positive when restricted on
$D_{2}$ unless $D=D_{1}$. Hence
\[
(v_{c})_{\alpha \bar{\beta}}=0
\]
on $M$ for all $t_{0}>0$ small enough. By the gradient estimate in \cite{bbg},
\cite{ckt} and (\ref{88}), we can conclude that $(v_{c})(x,t_{0})$ is a
constant, provided $t_{0}>0$ is small enough. Hence $u_{c}$ is constant. Since
$c$ is arbitrary , it shows that $u(x)$ is constant.
\end{proof}

\textbf{The proof of Corollary \ref{C1} : }

\begin{proof}
By Theorem \ref{A2}, we can solve $\sqrt{-1}\partial_{b}\overline{\partial
}_{b}u=\rho$ where $u$ satisfies (\ref{40}). From (\ref{74}), we have
$u(x)=o(\log r)$. We obtain that $u(x)$ is constant by the Lemma \ref{T61}.
\ Then the proof is completed.
\end{proof}
\begin{remark}
Let $(L,h^L)$ be a holomorphic line bundle over a non-compact complete k\"{a}hler manifold $\Omega$, where $h^L$ denotes the Hermitian fiber metric of $L$. Assume that the curvature $R^L$ induced by $h^L$ is positive. Then $M=\{v\in L^*: |v|^2_{h^{L^*}}=1\}$
is a noncompact Sasakian manifold. This is an interesting example to understand the relation between CR Poincar\'{e}-Lelong
equation on $M$ and Poincar\'{e}-Lelong equation on $\Omega$. We will explore this example in a forthcoming paper.
\end{remark}
\section{Acknowledgement} The authors would like to thank Professor Xiangyu Zhou and the Editorial Board for the invitation to contribute this paper to a special issue of Acta Mathematica Sinica in memory of Professor Qikeng Lu.

\bigskip

\appendix

\section{ \ }

We introduce some basic materials in a pseudohermitian $(2n+1)$-manifold
( see \cite{l1}, \cite{l2} for more details ). Let $(M,\xi)$ be a
$(2n+1)$-dimensional, orientable, contact manifold with contact structure
$\xi$. A CR structure compatible with $\xi$ is an endomorphism $J:\xi
\rightarrow \xi$ such that $J^{2}=-1$. We also assume that $J$ satisfies the
following integrability condition: If $X$ and $Y$ are in $\xi$, then so are
$[JX,Y]+[X,JY]$ and $J([JX,Y]+[X,JY])=[JX,JY]-[X,Y]$.

Let $\left \{  T,Z_{\alpha},Z_{\bar{\alpha}}\right \}  $ be a frame of
$TM\otimes \mathbb{C}$, where $Z_{\alpha}$ is any local frame of $T_{1,0}%
,\ Z_{\bar{\alpha}}=\overline{Z_{\alpha}}\in T_{0,1}$ and $T$ is the
characteristic vector field. Then $\left \{  \theta,\theta^{\alpha}%
,\theta^{\bar{\alpha}}\right \}  $, which is the coframe dual to $\left \{
T,Z_{\alpha},Z_{\bar{\alpha}}\right \}  $, satisfies
\[
d\theta=ih_{\alpha \overline{\beta}}\theta^{\alpha}\wedge \theta^{\overline
{\beta}}
\]
for some positive definite hermitian matrix of functions $(h_{\alpha \bar
{\beta}})$, if we have this contact structure, we call such $M$ a strictly
pseudoconvex CR $(2n+1)$-manifold.

The Levi form $\left \langle \ ,\  \right \rangle _{L_{\theta}}$ is the Hermitian
form on $T_{1,0}$ defined by
\[
\left \langle Z,W\right \rangle _{L_{\theta}}=-i\left \langle d\theta
,Z\wedge \overline{W}\right \rangle .
\]
We can extend $\left \langle \ ,\  \right \rangle _{L_{\theta}}$ to $T_{0,1}$ by
defining $\left \langle \overline{Z},\overline{W}\right \rangle _{L_{\theta}%
}=\overline{\left \langle Z,W\right \rangle }_{L_{\theta}}$ for all $Z,W\in
T_{1,0}$. The Levi form induces naturally a Hermitian form on the dual bundle
of $T_{1,0}$, denoted by $\left \langle \ ,\  \right \rangle _{L_{\theta}^{\ast}%
}$, and hence on all the induced tensor bundles. Integrating the Hermitian
form (when acting on sections) over $M$ with respect to the volume form
$d\mu=\theta \wedge(d\theta)^{n}$, we get an inner product on the space of
sections of each tensor bundle.

The pseudohermitian connection of $(J,\theta)$ is the connection $\nabla$ on
$TM\otimes \mathbb{C}$ (and extended to tensors) given in terms of a local
frame $Z_{\alpha}\in T_{1,0}$ by
\[
\nabla Z_{\alpha}=\theta_{\alpha}{}^{\beta}\otimes Z_{\beta},\quad \nabla
Z_{\bar{\alpha}}=\theta_{\bar{\alpha}}{}^{\bar{\beta}}\otimes Z_{\bar{\beta}%
},\quad \nabla T=0,
\]
where $\theta_{\alpha}{}^{\beta}$ are the $1$-forms uniquely determined by the
following equations:
\[
\begin{split}
d\theta^{\beta}  &  =\theta^{\alpha}\wedge \theta_{\alpha}{}^{\beta}
+\theta \wedge \tau^{\beta},\\
0  &  =\tau_{\alpha}\wedge \theta^{\alpha},\\
0  &  =\theta_{\alpha}{}^{\beta}+\theta_{\bar{\beta}}{}^{\bar{\alpha}}.
\end{split}
\]
We can write (by Cartan lemma) $\tau_{\alpha}=A_{\alpha \gamma}\theta^{\gamma}$
with $A_{\alpha \gamma}=A_{\gamma \alpha}$. The curvature of Webster-Stanton
connection, expressed in terms of the coframe $\{ \theta=\theta^{0}
,\theta^{\alpha},\theta^{\bar{\alpha}}\}$, is
\[
\begin{split}
\Pi_{\beta}{}^{\alpha}  &  =\overline{\Pi_{\bar{\beta}}{}^{\bar{\alpha}}
}=d\omega_{\beta}{}^{\alpha}-\omega_{\beta}{}^{\gamma}\wedge \omega_{\gamma}
{}^{\alpha},\\
\Pi_{0}{}^{\alpha}  &  =\Pi_{\alpha}{}^{0}=\Pi_{0}{}^{\bar{\beta}}=\Pi
_{\bar{\beta}}{}^{0}=\Pi_{0}{}^{0}=0.
\end{split}
\]
Webster showed that $\Pi_{\beta}{}^{\alpha}$ can be written
\[
\Pi_{\beta}{}^{\alpha}=R_{\beta}{}^{\alpha}{}_{\rho \bar{\sigma}}\theta^{\rho
}\wedge \theta^{\bar{\sigma}}+W_{\beta}{}^{\alpha}{}_{\rho}\theta^{\rho}
\wedge \theta-W^{\alpha}{}_{\beta \bar{\rho}}\theta^{\bar{\rho}}\wedge
\theta+i\theta_{\beta}\wedge \tau^{\alpha}-i\tau_{\beta}\wedge \theta^{\alpha},
\]
where the coefficients satisfy
\[
R_{\beta \bar{\alpha}\rho \bar{\sigma}}=\overline{R_{\alpha \bar{\beta}\sigma
\bar{\rho}}}=R_{\bar{\alpha}\beta \bar{\sigma}\rho}=R_{\rho \bar{\alpha}
\beta \bar{\sigma}},\  \  \ W_{\beta \bar{\alpha}\gamma}=W_{\gamma \bar{\alpha
}\beta}.
\]
Here $R_{\gamma}{}^{\delta}{}_{\alpha \bar{\beta}}$ is the pseudohermitian
curvature tensor, $R_{\alpha \bar{\beta}}=R_{\gamma}{}^{\gamma}{}_{\alpha
\bar{\beta}}$ is the pseudohermitian Ricci curvature tensor, $S=R_{\alpha
\overline{\alpha}}$ is the Tanaka-Webster scalar curvature and $A_{\alpha
\beta}$ \ is the torsion tensor. Furthermore, we define the bisectional
curvature
\[
R_{\alpha \bar{\alpha}\beta \overline{\beta}}(X,Y)=R_{\alpha \bar{\alpha}%
\beta \overline{\beta}}X_{\alpha}X_{\overline{\alpha}}Y_{\beta}Y_{\bar{\beta}}%
\]
and the bitorsion tensor
\[
T_{\alpha \overline{\beta}}(X,Y):=i(A_{\bar{\beta}\bar{\rho}}X_{\rho}Y_{\alpha
}-A_{\alpha \rho}X_{\bar{\rho}}Y_{\bar{\beta}})
\]
and the torsion tensor \
\[
Tor(X,Y):=h^{\alpha \bar{\beta}}T_{\alpha \overline{\beta}}(X,Y)=i(A_{\overline
{\alpha}\bar{\rho}}X_{\rho}Y_{\alpha}-A_{\alpha \rho}X_{\bar{\rho}}
Y_{\overline{\alpha}})
\]
for any $X=X_{\overline{\alpha}}Z_{\alpha},\ Y=Y_{\overline{\alpha}}Z_{\alpha
}$ in $T_{1,0}.$

We will denote components of covariant derivatives with indices preceded by
comma; thus write $A_{\alpha \beta,\gamma}$. The indices $\{0,\alpha
,\bar{\alpha}\}$ indicate derivatives with respect to $\{T,Z_{\alpha}
,Z_{\bar{\alpha}}\}$. For derivatives of a scalar function, we will often omit
the comma, for instance%
\[
u_{\alpha}=Z_{\alpha}u;\ u_{\alpha \bar{\beta}}=Z_{\bar{\beta}}Z_{\alpha
}u-\omega_{\alpha}{}^{\gamma}(Z_{\bar{\beta}})Z_{\gamma}u.
\]
For a smooth real-valued function $u$, the subgradient $\nabla_{b}$ is defined
by $\nabla_{b}u\in \xi$ and $\left \langle Z,\nabla_{b}u\right \rangle
_{L_{\theta}}=du(Z)$ for all vector fields $Z$ tangent to contact plane.
Locally
\[
\nabla_{b}u=\sum_{\alpha}u_{\bar{\alpha}}Z_{\alpha}+u_{\alpha}Z_{\bar{\alpha}}%
\]
and
\[
u_{0}=Tu.
\]
We can use the connection to define the subhessian as the complex linear map
\[
(\nabla^{H})^{2}u:T_{1,0}\oplus T_{0,1}\rightarrow T_{1,0}\oplus T_{0,1}%
\]
by
\[
(\nabla^{H})^{2}u(Z)=\nabla_{Z}\nabla_{b}u.\
\]
In particular,
\[
|\nabla_{b}u|^{2}=2u_{\alpha}u_{\overline{\alpha}},\quad|\nabla_{b}^{2}%
u|^{2}=2(u_{\alpha \beta}u_{\overline{\alpha}\overline{\beta}}+u_{\alpha
\overline{\beta}}u_{\overline{\alpha}\beta}).
\]
Also
\begin{eqnarray}\label{7.1}
\begin{array}
[c]{c}
\Delta_{b}u=Tr\left(  (\nabla^{H})^{2}u\right)  =\sum_{\alpha}(u_{\alpha
\bar{\alpha}}+u_{\bar{\alpha}\alpha}).
\end{array}
\end{eqnarray}

The Kohn-Rossi Laplacian $\square_{b}$ on functions is defined by
\[
\square_{b}\varphi=2\overline{\partial}_{b}^{\ast}\overline{\partial}
_{b}\varphi=(\Delta_{b}+inT)\varphi=-2\varphi_{\overline{\alpha}}{}
^{\overline{\alpha}}
\]
and on $(p,q)$-forms is defined by
\[
\square_{b}=2(\overline{\partial}_{b}^{\ast}\overline{\partial}_{b}
+\overline{\partial}_{b}\overline{\partial}_{b}^{\ast}).
\]

Next we recall the following commutation relations (\cite{l1}). \ Let
$\varphi$ be a scalar function and $\sigma=\sigma_{\alpha}\theta^{\alpha}$ be
a $\left(  1,0\right)  $ form, then we have
\begin{equation}
\begin{array}
[c]{ccl}
\varphi_{\alpha \beta} & = & \varphi_{\beta \alpha},\\
\varphi_{\alpha \bar{\beta}}-\varphi_{\bar{\beta}\alpha} & = & ih_{\alpha
\overline{\beta}}\varphi_{0},\\
\varphi_{0\alpha}-\varphi_{\alpha0} & = & A_{\alpha \beta}\varphi_{\bar{\beta}%
},\\
\sigma_{\alpha,0\beta}-\sigma_{\alpha,\beta0} & = & \sigma_{\alpha,\bar
{\gamma}}A_{\gamma \beta}-\sigma_{\gamma}A_{\alpha \beta,\bar{\gamma}},\\
\sigma_{\alpha,0\bar{\beta}}-\sigma_{\alpha,\bar{\beta}0} & = & \sigma
_{\alpha,\gamma}A_{\bar{\gamma}\bar{\beta}}+\sigma_{\gamma}A_{\bar{\gamma}%
\bar{\beta},\alpha},
\end{array}
\label{A}
\end{equation}
and
\[
\begin{array}
[c]{ccl}
\sigma_{\alpha,\beta \gamma}-\sigma_{\alpha,\gamma \beta} & = & iA_{\alpha
\gamma}\sigma_{\beta}-iA_{\alpha \beta}\sigma_{\gamma},\\
\sigma_{\alpha,\bar{\beta}\bar{\gamma}}-\sigma_{\alpha,\bar{\gamma}\bar{\beta
}} & = & ih_{\alpha \overline{\beta}}A_{\bar{\gamma}\bar{\rho}}\sigma_{\rho
}-ih_{\alpha \overline{\gamma}}A_{\bar{\beta}\bar{\rho}}\sigma_{\rho},\\
\sigma_{\alpha,\beta \bar{\gamma}}-\sigma_{\alpha,\bar{\gamma}\beta} & = &
ih_{\beta \overline{\gamma}}\sigma_{\alpha,0}+R_{\alpha \bar{\rho}}{}_{\beta
\bar{\gamma}}\sigma_{\rho}.
\end{array}
\]

Moreover for multi-index $I=\left(  \alpha_{1},...,\alpha_{p}\right)
,\  \bar{J}=\left(  \bar{\beta}_{1},...,\bar{\beta}_{q}\right)  ,$ we denote
$I(\alpha_{k}=\mu)=\left(  \alpha_{1},...,\alpha_{k-1},\mu,\alpha
_{k+1},...,\alpha_{p}\right)  .$ Then%
\[%
\begin{array}
[c]{ccl}%
\eta_{I\bar{J},\mu \lambda}-\eta_{I\bar{J},\lambda \mu} & = & i\sum_{k=1}%
^{p}\left(  \eta_{I(\alpha_{k}=\mu)\bar{J}}A_{\alpha_{k}\lambda}%
-\eta_{I(\alpha_{k}=\lambda)\bar{J}}A_{\alpha_{k}\mu}\right) \\
&  & -i\sum_{k=1}^{q}\left(  \eta_{I\bar{J}\left(  \bar{\beta}_{k}=\bar
{\gamma}\right)  }h_{\bar{\beta}_{k}\mu}A_{\lambda}^{\bar{\gamma}}-\eta
_{I\bar{J}\left(  \bar{\beta}_{k}=\bar{\gamma}\right)  }h_{\bar{\beta}%
_{k}\lambda}A_{\mu}^{\bar{\gamma}}\right)  ,
\end{array}
\]
and
\[
\begin{array}
[c]{ccl}
\eta_{I\bar{J},\lambda \bar{\mu}}-\eta_{I\bar{J},\bar{\mu}\lambda} & = &
ih_{\lambda \bar{\mu}}\eta_{I\bar{J},0}+\sum_{k=1}^{p}\eta_{I\left(  \alpha
_{k}=\gamma \right)  \bar{J}}R_{\alpha_{k}\mathit{\  \ }\lambda \bar{\mu}%
}^{\mathit{\  \  \ }\gamma}+\sum_{k=1}^{q}\eta_{I\bar{J}\left(  \bar{\beta}%
_{k}=\bar{\gamma}\right)  }R_{\bar{\beta}_{k}\mathit{\  \  \ }\lambda \bar{\mu}%
}^{\mathit{\  \  \  \ }\bar{\gamma}}\\
\eta_{I\bar{J},0\mu}-\eta_{I\bar{J},\mu0} & = & A_{\mu}^{\bar{\rho}}%
\eta_{I\bar{J},\bar{\rho}}-\sum_{k=1}^{p}A_{\alpha_{k}\mu,\bar{\rho}}%
\eta_{I\left(  \alpha_{k}=\rho \right)  \bar{J}}+\sum_{k=1}^{q}A_{\mu \rho
,\bar{\beta}_{k}}\eta_{I\bar{J}\left(  \bar{\beta}_{k}=\bar{\rho}\right)  }.
\end{array}
\]

\end{document}